\documentclass[11pt]{article}
\usepackage{amsmath, amsfonts, amsthm, amssymb, amscd}
\usepackage{amssymb, amscd, mathrsfs, wasysym,fullpage}
\numberwithin{equation}{section}  
\usepackage{graphicx}
        \newtheorem{theorem}{Theorem}[section]
        \newtheorem{proposition}[theorem]{Proposition}
        \newtheorem{lemma}[theorem]{Lemma}
        \newtheorem{corollary}[theorem]{Corollary} 
        \newtheorem{definition}[theorem]{Definition} 

\let\oldmarginpar\marginpar
\renewcommand\marginpar[1]{\-\oldmarginpar[\raggedleft\footnotesize #1]%
{\raggedright\footnotesize #1}}
\newcommand \uhat {\widehat u}

\newcommand \be {\begin{equation}}
\newcommand \bel {\be\label}
\newcommand \ee {\end{equation}}  
\newcommand	\RR 		{\mathbb R}   
\newcommand \del {{\partial}}
\newcommand \eps \epsilon 
\newcommand \loc {\text{loc}}
\newcommand \Mcal {\mathcal M} 
\newcommand \bei {\begin{itemize}}
\newcommand \eei {\end{itemize}}

\newcommand \trianglerightNEW \triangleright

\newcommand \rhot {\widetilde \rho}

\usepackage[usenames]{color}

\newcommand{\ONE}{{\mathbf{1}}}
\newcommand{\Fc}{\mathcal{F}}
\newcommand{\Pp}{\mathsf{P}}
\newcommand{\E}{\mathsf{E}}
\newcommand{\N}{\mathbb{N}}  
\newcommand{\Z}{\mathbb{Z}}
\newcommand{\R}{\mathbb{R}}
\newcommand{\U}{\mathbb{U}}
\newcommand{\V}{\mathbb{V}}
\newcommand{\D}{\mathbb{D}}

\newcommand{\Leb}{\mathrm{Leb}}
\newcommand{\sign}{\mathop{\mathrm{sgn}}}
\begin{document}

\title{Ergodicity of spherically symmetric fluid flows outside of a Schwarzschild black hole
with random boundary forcing
}

\author{Yuri Bakhtin\footnote{Courant Institute of Mathematical Sciences, New York University,  215 Mercer Street, New York, NY 10012, USA. Email: {\sl bakhtin@cims.nyu.edu.}} 
\, and 
%    Spelling:   LeFloch  or LeFLOCH
Philippe G. LeFloch\footnote{Laboratoire Jacques-Louis Lions \& Centre National de la Recherche Scientifique, 
Universit\'e Pierre et Marie Curie (Paris 6), 4 Place Jussieu, 75252 Paris, France. 
Email : {\sl contact@philippelefloch.org.} 
\newline
2000\textit{\ AMS Subject Class.} Primary: 35L65. Secondary: 76L05. 
}}

\date{May 2017}
\maketitle

\begin{abstract} We consider the Burgers equation posed on the outer communication region of a Schwarzschild black hole spacetime. Assuming spherical symmetry for the fluid flow under consideration, we study the propagation and  interaction of shock waves under the effect of random forcing. First of all, considering the initial and boundary value problem with boundary data prescribed in the vicinity of the horizon, we establish a generalization of the Hopf--Lax--Oleinik formula, which takes the curved geometry into account and allows us to establish the existence of bounded variation solutions. To this end, we analyze the global behavior of the characteristic curves in the Schwarzschild geometry, including their behavior near the black hole horizon. In a second part, we investigate the long-term statistical properties of solutions when a random forcing is imposed near the black hole horizon and study the  ergodicity of the fluid flow under consideration.  We prove the existence of a random global attractor and, for the Burgers equation outside of a Schwarzschild black hole, we are able to validate the so-called `one-force-one-solution' principle. Furthermore, all of our results are also established  for a pressureless Euler model which consists of two balance laws and includes a transport equation satisfied by the integrated fluid density. 
\end{abstract} 

\vfill
 
\setcounter{tocdepth}{1} 
\tableofcontents 

\newpage 

%================================================================

\section{Introduction and objectives}

We are interested in long-term statistical properties of nonlinear waves in compressible fluids under
stochastic forcing. Namely, we study ergodicity of such systems in the situation where random external influences leading to propagation and  interaction of radom shock waves are given by random boundary conditions.
An archetypal system of this kind is the stochastic Burgers equation, see a review on Burgers turbulence in~\cite{MR2318582}. Descriptions of ergodic properties
for stochastic Burgers equation and other conservation laws have been given for various situations with random volume forcing in 
\cite{ekms:MR1779561,
Iturriaga:MR1952472, Gomes-Iturriaga-Khanin:MR2241814, 
Suidan:MR2141893, 
Khanin-Hoang:MR1975784, B:MR3112935, Sinai:MR1117645, 
Kifer:MR1452549, 
BCK:MR3110798,
kickb:bakhtin2016, 
Bakhtin-Li:2016arXiv160704864B,
H-and-M:2016arXiv161003415H, Dirr-Souganidis:MR2191776, Debussche-Vovelle:MR3418750},
and with random boundary conditions on a bounded interval in~\cite{B:MR2299503}. 

The goal of this paper is to extend these  existing results in several directions. First, we study stationary solutions for the inviscid Burgers equation on the half-line
$\R_+=[0,+\infty)$ for random boundary conditions imposed at~$0$. Secondly, we study stationary solutions of a similar system in geometry with nontrivial curvature. More concretely, we work with the Schwarzschild geometry which is one of the simplest nontrivial geometries. It arises in Einstein's theory of general relativity and describes the exterior of a static and spherically symmetric vacuum black hole. In the Schwarzschild geometry case, 
the analysis turns out to be more involved due to the more complex structure and behavior of characteristics representing particle trajectories.  In contrast with the existing literature, our results are stated not just for 
the Burgers
or Burgers-type equation describing the velocity vector field of the (pressureless) fluid, but for the system of conservation laws
one of which is described above and the other one is a transport equation that governs the evolution of the fluid density.

The probabilistic analysis of the long-term properties of the system in the Schwarzschild metric would be hopeless without understanding the properties of solutions of the initial-boundary value problem.  In fact, a large part of the paper is devoted to the description of the Schwarschild geometry, the Burgers--Schwarschild model, a version
of the Hopf--Lax--Oleinik variational principle that can be used to construct solutions of the initial-boundary problem, 
and the behavior of the characteristics that solve the variational principle.  All these issues are of independent interest, so let us discuss them.

The Schwarzschild metric plays a central role in general relativity since it represents the {\sl only} possible final state of matter after a long-time evolution has taken place and the geometry has become static\footnote{The angular momentum and charge of the black hole could also be taken into account.}.  The presence of a black hole
means that there is a compact region that cannot be escaped by any signal.  
We pose the problem of propagation of nonlinear waves in the so-called exterior region outside 
the black hole horizon. This relativistic Burgers equation under consideration serves as a simplified model for relativistic fluid dynamics, and the existence of shock wave solutions was recently established in \cite{MR3145067,LMO,PLFSX}.

According to the Burgers--Schwarzschild equation under consideration, the fluid particles are attracted by the black hole so that the Lagrangian orbits realizing the minima in the Hopf--Lax--Oleinik variational principle
followed by the fluid particles are not straight lines. Moreover, they form a family of curves each of which either approaches the black hole horizon or converges to spacelike infinity.  This behavior leads to a need to revisit some of the well-known properties enjoyed by the standard Burgers equation, including the variational principle and its consequences for long-term properties of solutions. 

If $M$ denotes the mass of the black hole and standard Schwarzschild coordinates $(t,r)$ are used, the spatial variable describes $r \in [2M, +\infty)$ in which $r=2M$ represents the horizon radius (in appropriate units). Since no signal can escape from the horizon, it is natural to impose stationary random forcing at some radius $r=r_* > 2M$ which can be in an arbitrary vicinity of the horizon. Under rather general conditions on the random forcing, we establish the existence of a {\sl global solution} which can be viewed as a {\sl one-point   random pullback attractor on the main ergodic component} associated with the realization of the noisy
boundary conditions. We also investigate the basin of attraction associated with this solution. For the model under consideration, we thus validate the so-called {\sl one-force-one-solution principle} (1F1S) also known as {\sl synchronization by noise} that can be viewed as a manifestation of fast memory loss in the system. For a discussion of 1F1S in other models related to the Burgers equation, we can refer to lecture notes~\cite{Bakhtin:MR3526825} or to a more recent summary of known results in the introduction to~\cite{Bakhtin-Li:2016arXiv160704864B}.

The outline of this paper is as follows. In Section~\ref{sec2}, we present the fluid model of interest by writing the Schwarzschild metric and then expressing the (pressureless) Euler system in this geometry. In Section~\ref{sec3},  we generalize the Hopf--Lax--Oleinik formula for the initial-boundary value problem for the Burgers equation in a Schwarzschild background and study the associated characteristic curves. Existence theorems for integrable and non-integrable initial data are established here. Next, we turn to the analysis of the random forcing imposed on the boundary and, in Section~\ref{sec5}, first treat the standard Burgers equation corresponding to taking $M=0$ in our model. Finally, Section~\ref{sec6} is devoted to the Euler--Schwarschild model where a random boundary forcing is imposed near the horizon. Some concluding remarks are presented in Section~\ref{sec7}. 

\vskip.5cm 

\noindent{\bf Acknowledgments.} The work of YB was partially supported by the National Science Foundation (NSF) through grant DMS-1460595. This work was done when PLF was visiting the Courant Institute of Mathematical Sciences (NYU) and  was also partially supported by the Innovative Training Networks (ITN) grant 642768 (ModCompShock).

%===========================================================================

\section{Relativistic fluid models of interest}  
\label{sec2}

\subsection{The Schwarzschild background geometry}
 
We are interested in modeling the dynamics of a fluid outside a Schwarzschild black hole (see for instance~\cite{Hawking}). By using the so-called  Schwarzschild coordinates $t \geq 0$ and $r \in (2M, +\infty)$, the domain of outer communication of this spacetime
$$
\Mcal := \big\{ (t,r, \theta, \varphi) \, \ t \geq 0, \, r \geq 2M \big\} 
$$
is described by the metric
\bel{Scharz01} 
g = - \Big(1 - {2 M \over r} \Big) dt^2 + \Big(1 - {2 M \over r} \Big)^{-1} dr^2 + r^2 \, g_{S^2}, \qquad r>2M. 
\ee
Here, $g_{S^2}:=d\theta^2+(\sin \theta)^2 \, d\varphi^2$ (with variables $\theta \in [0, 2\pi)$ and $\varphi \in [0, \pi]$) denotes the canonical metric on the unit two-sphere $S^2$. The metric coefficients are singular when $r$ approaches $2M$, but this is not a genuine geometric singularity. The boundary $r=2M$ represents the ``horizon'' of a black hole region, and the apparent singularity in \eqref{Scharz01} is not of a geometric nature but could be removed by introducing a different coordinate chart. However, the fluid equations under consideration below would become algebraically much more involved, but doing so is unnecessary if one works ``outside'' the black hole, i.e. in the domain of outer communication.  

We consider the pressureless Euler equations 
\bel{re-Euler01}
\nabla^\alpha \big( T_{\alpha\beta}(\rho,u) \big) = 0,
\ee
in which $\nabla$ denotes the Levi-Civita connection associated with the metric $g$, 
and $T^\alpha_ \beta$ is the energy-momentum tensor of a pressureless  fluid. That is, the speed of light being normalized to unit, the matter tensor 
\bel{tensor-form01}
T_{\alpha\beta} (\rho,u) = \rho  u_\alpha u_\beta 
\ee
depends on the mass-energy density of the fluid denoted by $\rho: M \mapsto (0, +\infty)$ 
and its velocity field denoted by $u= (u^\alpha)$. The latter is normalized to be a unit and future-oriented vector field, that is,  
\bel{eq-u01} 
u^\alpha u_\alpha=-1, \qquad u^0 > 0. 
\ee   We use here standard notation for the metric $g= (g_{\alpha\beta})$ and its inverse $g^{-1} = (g^{\alpha\beta})$ in coordinates $x= (x^\alpha)$, where the Greek indices describe $0,1,2,3$. 
We raise and lower indices by using the metric and therefore, for instance, we write $u_\alpha = g_{\alpha\beta} u^\beta$ and $g^\alpha_\beta = \delta^\alpha_\beta$ (the Kronecker symbol).

By using $(0,1,2,3)$ to denote the coordinates $(t,r,\theta,\varphi)$, the metric reads 
$$
(g_{\alpha\beta}) = \text{diag} \Big(-(1-2M/r), (1-2M/r)^{-1},  r^2,  r^2(\sin \theta)^2 \Big), 
$$
while its inverse is 
$$
(g^{\alpha\beta}) = \text{diag} \Big(-(1-2M/r)^{-1}, (1-2M/r),  r^{-2},  r^{-2}(\sin \theta)^{-2} \Big).
$$
After a tedious calculation, the non-vanishing Christoffel symbols 
$\Gamma_{\alpha\beta}^\gamma := {1\over 2}g^{\gamma\theta}(\del_\alpha g_{\beta\theta}+\del_ \beta g_{\alpha\theta} -\del_\theta  g_{\alpha\beta})$ are  found to be 
\be
\aligned
& \Gamma_{00}^1= {M\over r^2}(r-2M), \qquad 
&&\Gamma_{11}^1= - {M\over r(r-2M)},
\qquad
&&&&\Gamma_{01}^0= {M\over r(r-2M)},
\\
& \Gamma_{12}^2= {1\over r}, \qquad
&&\Gamma_{22}^1=-(r-2M),
\qquad 
&&&&\Gamma_{13}^3= {1\over r},
\\
& \Gamma_{33}^1=-(r-2M)(\sin \theta)^2, \qquad
&&\Gamma_{33}^2=-\sin\theta\cos\theta, 
\qquad
&&&&\Gamma_{23}^3= {\cos\theta\over \sin\theta}.
\endaligned
\ee
  
%-------------------------------------------------------------------------------------------------------------------------------

\subsection{The relativistic Euler system} 

Throughout, we assume that the fluid flow is spherically symmetric, so that the non-radial components of the fluid velocity vanish: 
\be
T^{02} =T^{03} = 0.
\ee
By expressing the Euler equations \eqref{re-Euler01} in the form 
$$
\aligned
\del_t T^{0 \beta}+\del_jT^{j\beta}+\Gamma_{00}^0T^{0 \beta}
& +\Gamma_{j0}^jT^{\beta0}+\Gamma_{0j}^0T^{j\beta}+\Gamma_{jk}^jT^{k\beta}
\\
&
  +\Gamma_{00}^\beta T^{00}+2\Gamma_{j0}^\beta T^{j0}+\Gamma_{jk}^\beta T^{jk} = 0, 
\endaligned
$$
where a tedious calculation shows that 
 the Euler equations on a Schwarzschild background read  
\bel{Euler001}
\aligned
&\del_t \Big( r(r-2M)T^{00} \Big) + \del_r \Big( r(r-2M)T^{01} \Big) = 0,
\\ 
&\del_t \Big( r(r-2M)T^{01} \Big) + \del_r  \Big( r(r-2M)T^{11} \Big) = \Omega_1, 
\\
& \Omega_1 := 3MT^{11} - {M \over r^2} (r-2M)^2T^{00} + r(r-2M)^2T^{22} + r(\sin \theta)^2 \, (r-2M)^2T^{33}.
\endaligned
\ee
With the expression \eqref{tensor-form01} of the energy-momentum tensor, we thus obtain  
\bel{EulerC01}  
\aligned
& \del_t \Big(  r (r - 2M) \rho u^0u^0 \big) \Big)
 + \del_r \Big(  r(r-2M) \rho u^0u^1 \Big) = 0,
\\
& \del_t \Big(  r^2 (1 - 2M/r) \rho u^0u^1 \Big) 
  + \del_r \Big(  r^2 (1-2M/r)^{-1}\rho u^1u^1 \Big) 
= \Omega_1, 
\\
&\Omega_1 =   3M  (1-2M/r)^2 \rho u^0u^0  
   - {M \over r^2}  \rho  u^1u^1.
\endaligned
\ee

The fluid velocity, by definition, is a future-oriented, unit timelike vector so that 
\be
(1-2M/r) u^0 u^0- (1-2M/r)^{-1}u^1 u^1=1, 
\qquad
u^0 > 0, 
\ee
while the energy-momentum tensor read 
$$
 \aligned
& T^{00}  = (1-2M/r)^2 \rho u^0 u^0, 
\quad &&T^{01} = \rho u^0 u^1, 
\\
& T^{11} =  (1-2M/r)^{-2} \rho u^1 u^1, \quad  &&T^{22} =T^{33} = 0.
\endaligned
$$  
Moreover, it is convenient to introduce  the real-valued function 
\bel{eq-def-scalV01}
u:= {1 \over (1 - 2M/r)} {u^1 \over u^0}, 
\ee
so that 
$$
(u^0)^2 = {1 \over (1 - u^2)(1-2M/r)},
\qquad (u^1)^2= (1-2M/r){u^2 \over 1- u^2}.
$$
Hence, we have formulated the Euler system for pressureless fluids on a Schwarzschild background, that is, the following system of two balance laws with unknowns $\rho \geq 0$ and $u \in (-1, 1)$
\bel{Euler01}
\aligned 
& \del_t \Big( r^2 {\rho \over 1 - u^2} \Big) +\del_r \Big(r(r-2M) {\rho \over 1 - u^2} u \Big) = 0,
\\
& \del_t \Big(r(r-2M) {\rho \over 1 - u^2} u \Big) + \del_r \Big((r-2M)^2 {\rho u^2 \over 1 - u^2} \Big)= \Omega(\rho,u,r), 
\\
& \Omega(\rho,u,r) := 3M \Big( 1 - {2M \over r} \Big) {\rho u^2 \over 1 - u^2}  -M {r-2M \over r} {\rho \over 1 - u^2}, 
\endaligned 
\ee
in which  $M\geq 0$ denotes the mass of the Schwarzschild black hole and the speed of light has been normalized to unit. Some simplification will be made in order to derive our model of interest in the present paper. 

Observe that in the limit of vanishing mass $M \to 0$, this system simplifies drastically and becomes
\bel{Euler01-vanishingM}
\aligned 
& \del_t \Big( r^2 {\rho \over 1 - u^2} \Big) +\del_r \Big( r^2 {\rho \over 1 - u^2} u \Big) = 0,
\\
& \del_t \Big( r^2 {\rho \over 1 - u^2} u \Big) + \del_r \Big( r^2 {\rho u^2 \over 1 - u^2} \Big)= 0, 
\endaligned 
\ee
which is equivalent to the non-relativistic Euler system for pressureless fluids in a flat geometry 
\bel{Euler01-vanishingM-standard}
\aligned 
& \del_t \big( \rhot \big) +\del_r \big( \rhot u \big) = 0,
\\
& \del_t \big( \rhot u \big) + \del_r \big( \rhot u^2 \big)= 0, 
\endaligned 
\ee
provided one introduces the normalized density variable
\bel{eq:nfd}
\rhot := {r^2 \rho \over 1 - u^2}. 
\ee

%-------------------------------------------------------------------------------------------------------------------------------------- 

\subsection{The Burgers--Schwarzschild and Euler--Schwarzschild models} 

For sufficiently smooth solutions bounded away from the vacuum, we can formally manipulate the equations in \eqref{Euler01} and naturally introduce the following definition. 

\begin{definition} The {\em Burgers--Schwarzschild model} is the following balance law 
\bel{eq:Bmodel}
\aligned
& \del_t u + \Big( 1 - {2M \over r} \Big) \del_r \Big( {u^2 \over 2} \Big) 
= {M \over r^2} ( u^2 - 1 ),  
\\
&  u=u(t,r) \in (-1, 1), \qquad \, t \geq 0, \, r >2M. 
\endaligned
\ee
 \end{definition}

Recall that $M>0$ represents the mass of the black hole and the speed of light is normalized to unit. 
Next, returning to the full Euler system and being motivated by \eqref{eq:nfd}, we introduce the {\sl integrated fluid density} (i.e. the total mass of matter contained within a region $[2M, r]$)  
\be
v(t,r) := \int_{2M}^r {\rho(t,r') \over 1 - u^2(t,r')} \, {r'}^2  dr' 
\ee
and obtain 
$
\del_t \del_r v + \del_r \Big( \big( 1- {2M \over r} \big) u \del_r v \Big) = 0. 
$
After integration in space and by observing that the coefficient $\big( 1- {2M \over r} \big)$ vanishes on the horizon $r=2M$, we are naturally led to the following definition. 

\begin{definition} The (pressureless) {\em Euler--Schwarzschild model} is the following system of two balance laws with unknowns $u=u(t,r) \in (-1, 1)$ and $v=v(t,r) \geq 0$
\bel{eq:ES} 
\aligned
& \del_t u + \Big( 1 - {2M \over r} \Big) \del_r \big( {u^2 \over 2} \Big) 
= {M \over r^2} ( u^2 - 1),  \qquad  r >2M,
\\
&\del_t v + \big( 1- {2M \over r} \big) u \del_r v = 0. 
\endaligned
\ee 
 \end{definition}

Clearly, when $M \to 0$, in \eqref{eq:Bmodel}  we recover the standard Burgers equation, which is well recognized as a fundamental model for nonlinear wave propagation in (non-relativistic) fluids. For $M\neq 0$, we refer to LeFloch and co-authors in \cite{MR3145067,LMO,PLFSX} for the existence of weak solutions to the Cauchy problem, when the initial data have bounded variation.

%==================================================================

\section{A Hopf--Lax--Oleinik formula for the Burgers--Schwarz\-schild model}
\label{sec3}

\subsection{The geometry of the characteristic curves}

Our first objective is to present a suitable generalization of the classical Hopf--Lax--Oleinik formula for the initial value problem when some initial data is imposed at some constant time and a boundary data is prescribed near the horizon. We begin by analyzing the equations satisfied by a solution to \eqref{eq:Bmodel} along a characteristic curve, i.e.  
\bel{eq:charact}
\aligned
r'(t) &= \Big( 1 - {2M \over r} \Big) u, 
\qquad 
u'(t) = {M \over r^2} ( u^2 - 1)
\endaligned
\ee 
with data $r(t_0)=r_0$ and $u(t_0)=u_0$ prescribed at some time $t_0$. 
It is convenient to first look for $u=u(r)$ {\sl as a function of $r$,} hence  
$$
-{2u \over 1 - u^2} {du \over dr} = {2M \over r(r-2M)}. 
$$
This implies (cf. Figure~\ref{u-as-function-r} for a plot for various values $C$) 
\begin{equation}
\label{eq:integral-of-motion}
\frac{1-u(r)^2}{1 - {2M \over r}} = C 
\end{equation}
for some constant $C$ and, by taking the initial data into account, we obtain the unknown velocity as a function of the radius variable
\bel{eq:31} 
{1-u(r)^2 \over 1- u_0^2} = {1 - {2M \over r} \over 1 - {2M \over r_0}}.  
\ee
On the other hand, the radius function solves the ODE 
\bel{equa33}
r'(t) = \pm \Big( 1 - {2M \over r} \Big) \sqrt{1 -  {1 - {2M \over r} \over 1 - {2M \over r_0}} (1- u_0^2)}, 
\ee
in which the sign depends upon the sign of the velocity given by \eqref{eq:31} and may change along a given characteristic curve.

\begin{figure}[htbp]
\begin{centering} 
\includegraphics[width=8cm,height=8cm,angle=0]{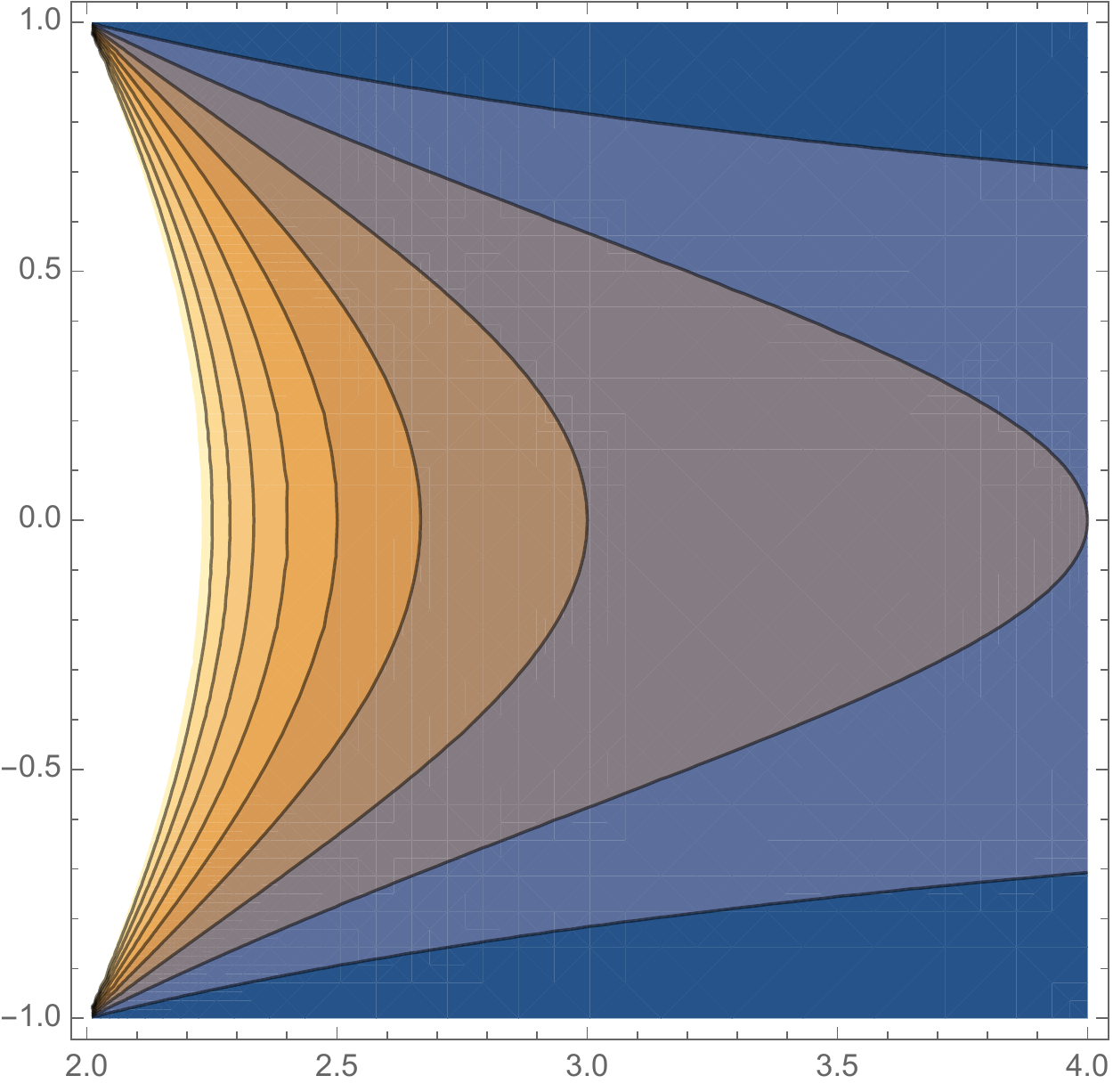}  
\par\end{centering}
\caption{Two plots (space vs. time) of the characteristic curves emanating from a point $r_0$}
\label{u-as-function-r} 
\end{figure}

It is important to introduce the {\sl escape velocity} $u_0^E$ by the condition that the right-hand side of \eqref{equa33} vanishes as $r \to + \infty$, thus
\be
u_0^E := \sqrt{2M \over r_0}.
\ee 
With this notation, \eqref{eq:31} and \eqref{equa33} read 
\bel{eq:31-escape} 
u(r)^2 = u_0^2 \Bigg( 1 + {(1-u_0^2)(u_0^E)^2 \over (1-(u_0^E)^2) u_0^2} \Big( 1 - {r_0 \over r} \Big)\Bigg),   
\ee
\bel{equa33-escape}
r'(t) = \pm \Big( 1 - {2M \over r} \Big) \sqrt{1 -  {1- u_0^2 \over 1 - (u_0^E)^2}  \Big(1 - {2M \over r} \Big)}. 
\ee
We have the following global behavior of the characteristic curves: 

\bei 

\item {\em Negative initial data}. In view of the second equation in \eqref{eq:charact} and since $u \in (-1, 1)$, the solution $u(t)$ is decreasing in the future time direction. If the initial data at $(t_0, r_0)$ is negative, that is if $u_0 \in (-1,0)$, then the solution $u=u(t)$ remains negative for all future times and the characteristic $r=r(t)$ propagates toward the black hole and asymptotically reaches the horizon in the future direction, while the solution approaches minus the speed of light (normalized to be $1$): 
\be
\lim_{t \to +\infty} r(t) = 2M, \qquad \lim_{t \to + \infty} u(t) = -1.
\ee

\item {\em Positive initial data.} If the initial data is non-negative, that is, $u_0 \in [0, 1)$,  then the characteristic $r=r(t)$  is initially {\sl moving away} from the black hole but, according to the equation $u'(t) = {M \over r^2} ( u^2 - 1 )$, the velocity decreases and we can distinguish between two distinct behaviors: 

\bei

\item {\em Initial velocity below the escape velocity:}  $u_0 < u_0^E$.
Then, there exists a finite time $t_1 >t_0$ after which $u(t) <0$ is negative and the phenomena in the previous case takes place. That is, from $t_1$ onwards, the characteristic moves toward the black hole and asymptotically
\be
\lim_{t \to +\infty} r(t) = 2M, \qquad \lim_{t \to + \infty} u(t) = -1.
\ee

\item {\em Initial velocity above the escape velocity:}  $u_0 \geq u_0^E$, in which case the characteristic moves away from the horizon for all times and never come back. Asymptotically we obtain 
\be
\label{eq:limiting-speed}
\lim_{t \to +\infty} r(t) = +\infty, 
\qquad 
\lim_{t \to + \infty} u(t) = u^\sharp_\infty := \sqrt{u_0^2 - (u_0^E)^2 \over 1 -(u_0^E)^2} \in [0, 1).
\ee

\eei

\eei 
We can also see that the conserved quantity $C$ introduced in~\eqref{eq:integral-of-motion} satisfies 
\begin{equation}
\label{eq:conserved-quantity-0}
C=\frac{u^2(r_*)-u^2_E(r_*)}{1-u^2_E(r_*)}=\frac{u^2(r)-\frac{2M}{r}}{1-\frac{2M}{r}}\in 
\left[-\frac{\frac{2M}{r_*}}{1-\frac{2M}{r_*}},1\right]= \left[-\frac{2M}{r_*-2M},1\right],
\end{equation}
and, in addition, negative values of $C$ define bounded characteristics, and nonnegative values of $C$ define unbounded  characteristics. In the latter case, the limiting value of velocity obtained in~\eqref{eq:limiting-speed} satisfies
$u^\sharp_\infty=\sqrt{C}$.

A similar dichotomy holds in the past direction $t \to -\infty$.
Clearly, the behavior above is  very much in contrast with the properties enjoyed by the standard Burgers equation, for which characteristics are {\sl straight-lines} and the solution is a {\sl constant} along them. 
This description is illustrated in Figure~1, in which any characteristic extends from the initial point $(t_0, r_0)$ in both (future and past) time directions. 
Observe also that the convergence of $r(t)$ toward 2M is exponential
\be
r(t) \simeq 2M + K e^{-|t|/(2M)} \quad \text{ as } t \to \pm\infty. 
\ee

\begin{figure}[htbp]
\begin{centering} 
\includegraphics[width=8cm,height=8cm,angle=0]{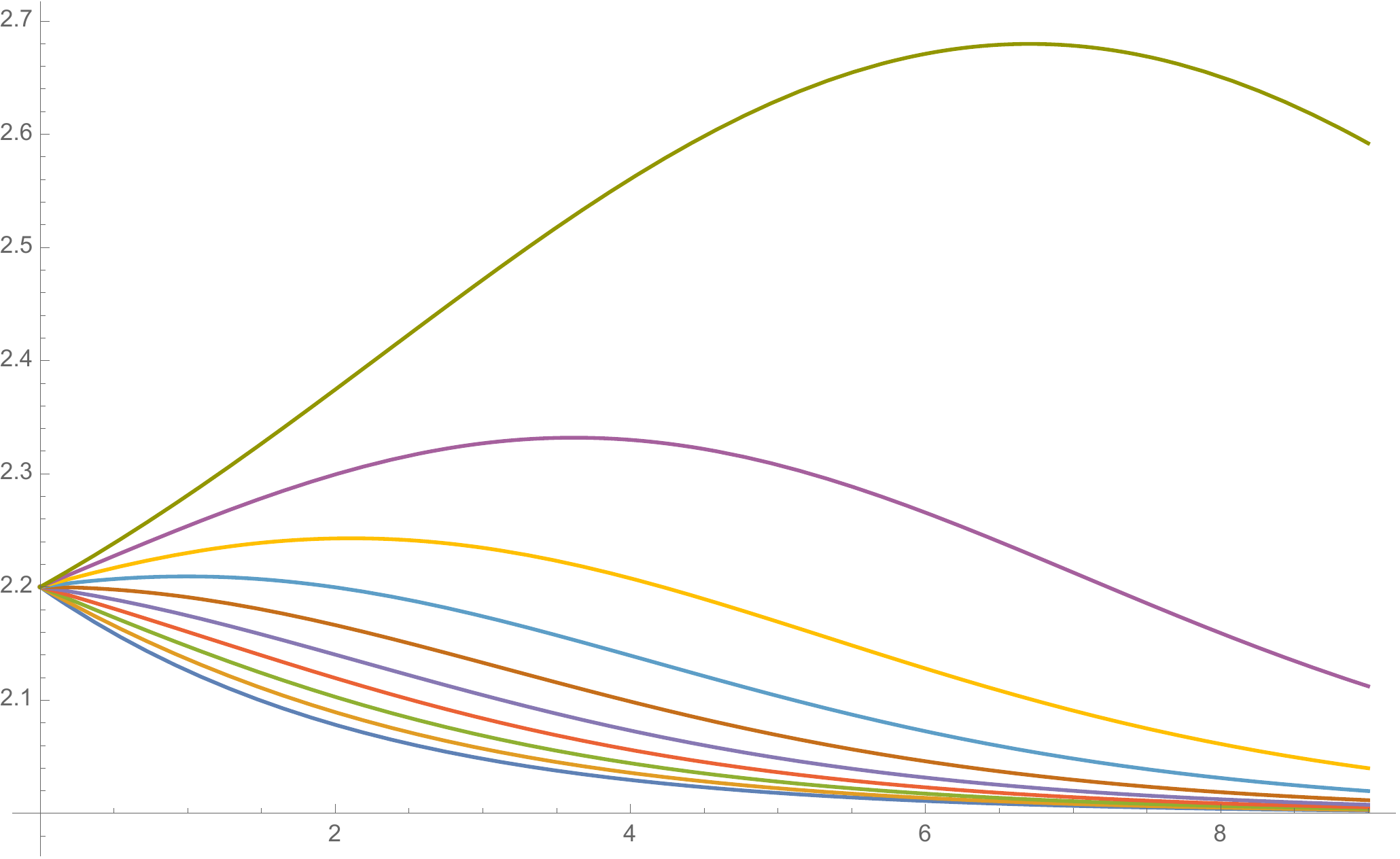} 
\includegraphics[width=8cm,height=8cm,angle=0]{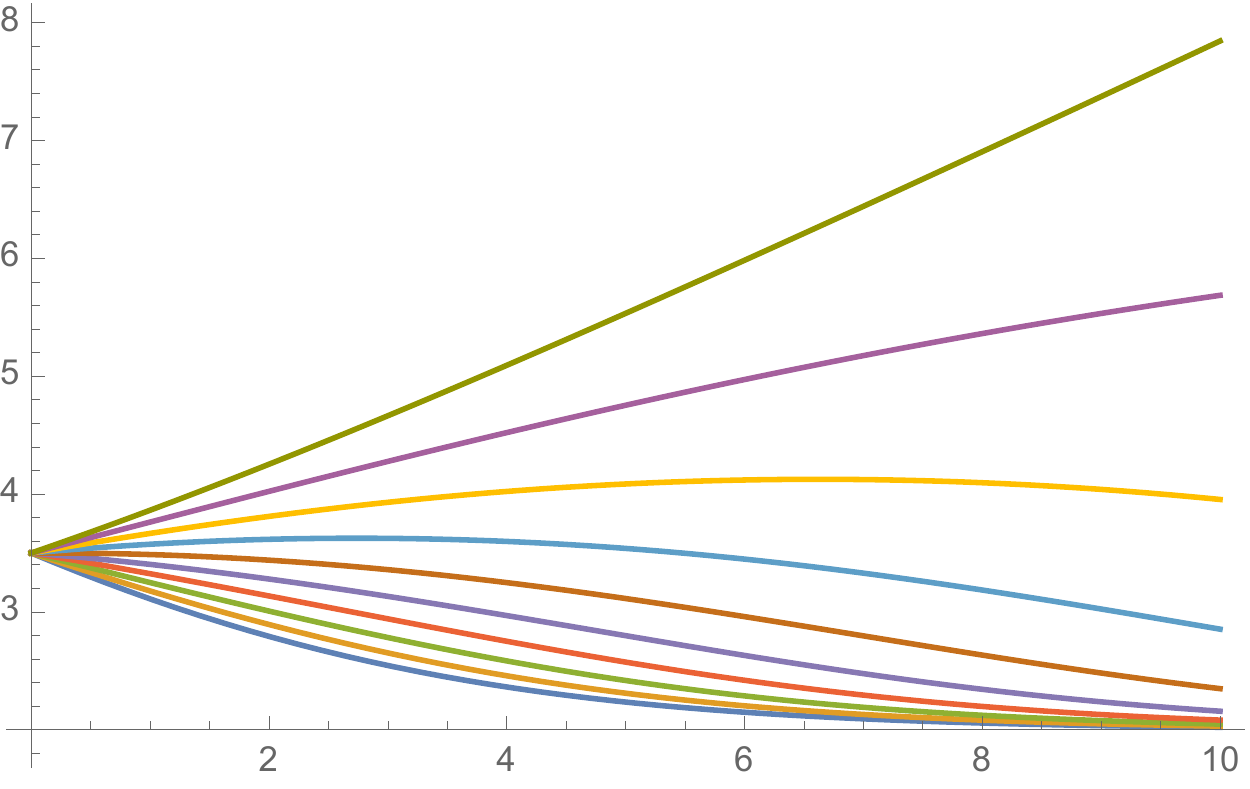} 
\par\end{centering}
\caption{Two plots (space vs. time) of the characteristic curves emanating from a point $r_0$}
\label{Fig1} 
\end{figure} 

%--------------------------------------------------------------------------

It is interesting to consider the limiting case when $u_0=\pm 1$ in which case the characteristic satisfies
$r'(t) = \pm \Big( 1 - {2M \over r} \Big)$, which can be integrated as
\be
e^{r(t) - r_0} \Bigg({ r(t) - 2M \over r_0 - 2M} \Bigg)^{2M} = e^{\pm(t-t_0)} \qquad \text{ when } u_0 = \pm 1.
\ee
These two curves are asymptotic, in the past or future direction, to straight-lines with slope $\pm 1$, and determine the boundary of the spacetime domain which can be attained from a point $(r_0, t_0)$ 
and is nothing but the light cone from this point.   

In comparison with the standard Burgers equation, one significant difference lies in the {\sl upper bound assumed on the propagation speed,} that is, the condition $|u| <1$ or by extension $|u| \leq 1$ if the light cone is included.  No such restriction is imposed in the standard Burgers equation so that, a priori, any two points $(t_0, r_0)$ and $(t_1, r_1)$ can possibly be connected by a particle trajectory, while this is possible in our relativistic model only if the light cones of these two points have a non-empty intersection. 
We refer to Figure~\ref{Fig1} for two typical plots of the characteristic funnel emanating from a point $r_0$. 

%---------------------------------------------------------------------------------------------------------------------

\subsection{Formulation of the problem} 

We now extend the Hopf--Lax--Oleinik formula by following the treatment of weighted conservation laws in spherical symetry, say of the form 
$\del_t (r^2 u) + \del_r (r^2 u^2/2) = 0$, made by LeFloch and Nedelec \cite{LeFlochNedelec}
and the treatment of the boundary condition in Joseph \cite{Joseph}. For a different approach to the boundary condition, we also refer to \cite{MR949657}. We also recall that, due to the presence of shock waves interacting with the boundary, weak solutions to noninear hyperbolic equations satisfy their boundary data only in some weak sense; see \cite{Bardos-leRoux-Nedelec:MR542510} for details. 

We observe that our equation admits the {\sl conservation form} 
\bel{Conservation}
\del_t \Bigg( \Big(1 - {2M \over r} \Big)^{-2} u   \Bigg)
+ {1\over 2}  \del_r  \Bigg(  \Big(1 - {2M \over r} \Big)^{-1}  \big(u^2 - 1 \big) + 1  \Bigg)
= 0, 
\ee
which, clearly, is valid even for weak solutions. 
We are interested in the associated initial and boundary value problem formulated in the half-space region of the form $r>r_*$ with $r_* >2M$ fixed. 
We prescribe an initial data $u_0=u_0(r)$ for $r > r_*$ on the hypersurface of constant time $t=t_0$, together with a boundary condition $u_*=u_*(t)$ at $r_*$, hence
\be
u|_{t=t_0} = u_0, \qquad u|_{r=r_*} = u_*, \qquad r >r_*.
\ee
Our construction below is based on the the funnel of characteristic curves $\chi=\chi(t)=\chi(t; t_1, r_1, u_1)$
associated with any initial point $(t_1, r_1)$ and defined by the backward ODE problem 
\be
\aligned
\chi'(t)  
& = \pm \Big( 1 - {2M \over \chi(t)} \Big) \sqrt{1 -  {1 - {2M \over \chi(t)} \over 1 - {2M \over r_1}} (1- u_1^2)},   
\\
\chi(t_1) & = r_1, 
\endaligned
\ee
in which $t \in (-\infty, t_1]$. Here, $u_1 \in (-1, 1)$ is a parameter representing the initial velocity at $(t_1, r_1)$.

To an arbitrary solution $u=u(t,r)$, we associate the function
\bel{eq:defU}
U(t,r) := - \int_r^{+\infty} \Big(1 - {2M \over r'} \Big)^{-2} u(t,r') \, dr'. 
\ee  
By integrating \eqref{Conservation} over any interval $[r, +\infty] \subset [r_*, +\infty)$ and assuming that the solution approaches zero when $r \to + \infty$, we obtain 
$$
\del_t U + {1\over 2}  
 \Big(1 - {2M \over r} \Big)^{-1}    \Big( u^2 - 1 \Big) +  {1 \over 2}
= 0. 
$$ 
Then, given any real $w$ and using the inequality 
$u^2 \geq w^2 + 2 w (u - w)$, we find 
\be
\aligned
& \del_t U + \Big( 1 - {2M \over r} \Big) w \, \del_r U 
\\&
\leq \Big( 1 - {2M \over r} \Big) w \, \del_r U 
- {1\over 2}  
 \Big(1 - {2M \over r} \Big)^{-1}    \Big( w^2 + 2 w (u - w) - 1 \Big) -  {1 \over 2},
% - f_*, 
\endaligned 
\ee
that is, 
\bel{eq:339}
 \del_t U + \Big( 1 - {2M \over r} \Big) w \, \del_r U 
\leq {1\over 2}  
 \Big(1 - {2M \over r} \Big)^{-1}    \Big( w^2 + 1 \Big) -  {1 \over 2}. 
%- f_*,
\ee 
The constant $1/2$ has been added throughout for convenience, so that we can easily recover the standard Burgers equation when $M$ approaches zero. 

We now integrate the inequality \eqref{eq:339} along an arbitrary characteristic denoted by $\chi=\chi(t)=\chi(t; t_1, r_1, u_1)$ within the time interval $t \in [t_0, t_1]$ and for the state $w$ required in \eqref{eq:339} we choose precisely the value $w=\varphi(t)=\varphi(t; t_1, r_1, u_1)$ achieved along this characteristic and defined in \eqref{eq:31}, that is, 
\be
{1-\varphi(t)^2 \over 1- u_1^2} = {1 - {2M \over \chi(t)} \over 1 - {2M \over r_1}}.  
\ee
Our next task will be to integrate the following inequality 
over the time interval $[t_0, t_1]$: 
\bel{eq:341}
\aligned
{d \over dt} U(t, \chi(t)) 
& = 
 \del_t U + \Big( 1 - {2M \over \chi(t)} \Big) \varphi(t) \, \del_r U 
\\
& \leq  {1\over 2}  
 \Big(1 - {2M \over \chi(t)} \Big)^{-1}    \Big( \varphi(t)^2 + 1 \Big) -  {1 \over 2}.
% - f_*(t)
\endaligned
\ee

\subsection{The generalized Hopf--Lax--Oleinik formula}

1. We first consider characteristics that do not hit the boundary $r_*$ (within the time interval under consideration) and write 
$$
U(t_1, r_1) \leq
 U(t_0, r_0) 
+  {1\over 2}  \int_{t_0}^{t_1} \Big( 
 \Big(1 - {2M \over \chi(t)} \Big)^{-1}    \big( \varphi(t)^2 + 1 \big) - 1 \Big) \, dt.
$$
Therefore, {\sl as long as} the characteristics do not meet the boundary $r=r_*$, we conclude that 
$$
U(t_1, r_1) 
= \min_{u_1 \in (-1,1)} \Bigg( 
 U(t_0, x_0) + \int_{t_0}^{t_1} F(t) \, dt 
% - \int_{t_0}^{t_1} f_*(t) \, dt
 \Bigg), 
$$
%in which the boundary flux was already defined by  
%$$
%2 \, f_*(t) :=  
% \Big(1 - {2M \over r_*} \Big)^{-1}    \Big( u_*(t)^2 - 1 \Big) + 1 
%$$
in which the (interior) flux $F(t)= F(t; t_1, r_1, u_1)$ is
$$
2 \, F(t) :=  \Big(1 - {2M \over \chi(t)} \Big)^{-1}    \big( \varphi(t)^2 + 1 \big) - 1. 
$$
Obviously, when the mass $M$ vanishes, this is nothing by the Hopf--Lax--Oleinik's expression for Burgers equation. This completes the description of our explicit formula whenever the minimizing characteristic would not meet the boundary. 

Before we proceed further, let us observe that, in the expression of $F(t)$, we can eliminate the state value $\varphi(t)$ given by 
\be
\varphi(t)^2 = 
1- (1- u_1^2) {1 - {2M \over \chi(t)} \over 1 - {2M \over r_1}}. 
\ee
Indeed, we find 
$$
\aligned
F(t) 
&= - {1 \over 2} + \Big(1 - {2M \over \chi(t)} \Big)^{-1} \Bigg(
           1 - {1 \over 2} (1- u_1^2) {1 - {2M \over \chi(t)} \over 1 - {2M \over r_1}}
\Bigg) 
\\
&=  {1 \over 2} \Bigg( {u_1^2 - 1 \over 1 - {2M \over r_1}} +1 \Bigg) 
   + {2M \over \chi(t) - 2M} =: F(\chi(t), r_1, u_1) 
\endaligned
$$
and, since the first term is a constant in $t$, 
\be
\label{eq:variational-principle-for-Schw-def-of-F}
 \int_{t_0}^{t_1} F(\chi(t)) \, dt 
= {(t_1 - t_0) \over 2} \Bigg( {u_1^2 - 1 \over 1 - {2M \over r_1}} +1 \Bigg) 
   +  2M \int_{t_0}^{t_1} {dt \over \chi(t) - 2M}.
\ee
The first term above depends explicitely on the velocity $u_1$ at the reference point $(t_1, r_1)$ and, in the limit of vanishing mass, approaches the integrated flux $(t_1-t_0)u_1^2/2$ arising in the standard formula.
On the other hand, the second term requires a non-trivial integration along the backward characteristic from $(t_1, r_1)$ and, in the limit of vanishing mass, tends to zero. 

%-----------------------------------------------------------------------------------------------------------------

2. We now take into account the possibility that the characteristic would hit the boundary, and by following the methodology in Joseph \cite{Joseph}, we arrive at the following formula 
%(with $r_1 > r_*$) 
\be
\label{eq:variational-principle-for-Schw}
U(t_1, r_1) 
= \min_{\chi}
%=\chi(t, r_1, v_1) \atop v_1 \in (-1,1)} 
\Bigg( 
 U_0(\chi(t_0)) + \int_{t_0}^{t_1} F(\chi(t)) \, \ONE_{\chi(t) \in (r_*, +\infty)} \, dt 
- \int_{t_0}^{t_1} f_*(t) \, \ONE_{\chi(t)=r_*} \, dt
 \Bigg), 
\ee
in which: 

\bei 

\item $U(t,r) := - \int_r^{+\infty} \Big(1 - {2M \over r'} \Big)^{-2} u(t,r') \, dr'$. 

\item $U_0$ is associated with the data $u_0: [r_*, +\infty) \to (-1, 1)$ prescribed at the initial time $t_0$.

\item $f_*(t)$ is the boundary flux determined from the positive part of the boundary data $u_*: [t_0, +\infty) \to (-1, 1)$ prescribed at $r_*$ 
\bel{eq-fluxstar}
2 \, f_*(t) :=  
 \Big(1 - {2M \over r_*} \Big)^{-1}    \Big( u_*^+(t)^2 - 1 \Big) + 1. 
\ee

\item 
The infimum is taken over all Lipschitz continuous curves $\chi=\chi(t)=\chi(t; t_1, r_1, u_1)$ satisfying 
$$
\chi(t_1)  = r_1
$$
and consisting of finitely many parts parametrized by time intervals 
$[t_0, t_1] = \bigcup_{i=1}^{2N+1} I_i$,  
within which 

\bei 

\item either the path coincides with the boundary (for $t \in I_{2p}$ and $p \leq N$): 
\be
\chi(t) = r_*, 
\ee

\item 
or else the path solves the characteristic equation (for $t \in I_{2p+1}$ and $p \leq N$): 
\be
\aligned
\chi'(t)  
& = \pm \Big( 1 - {2M \over \chi(t)} \Big) \sqrt{1 -  {1 - {2M \over \chi(t)} \over 1 - {2M \over r_1}} (1- u_\sharp^2)}
\endaligned
\ee
for some $u_\sharp \in (-1,1)$ depending on $p$, and is ``maximal'' in the sense that its endpoints on the interval $I_{2p+1}$ belong to the lines $\big\{t=t_0\big\}$  or $\big\{t=t_1\big\}$ or $\big\{r=r_*\big\}$.  

\eei 
\item 
Moreover, the flux $F$ is given by 
\bel{eq:flux-on-characteristic}
F(\chi(t)) =  {1 \over 2} \Bigg( {u_1^2 - 1 \over 1 - {2M \over r_1}} +1 \Bigg) 
   + {2M \over \chi(t) - 2M}.
\ee
\eei 

%-----------------------------------------------------------------------------------------------------------------

\subsection{The initial and boundary value problem}
 
In order to handle the above explicit formula, we need to control the behavior of the solutions of the ODE for $\chi(t)$. One should treat with care the integral term containing $1/(\chi(t) - 2M)$ which tends to infinity near the horizon of the black hole. The following observations are in order: 

\bei 

\item 
Recall that the characteristics have two possible behaviors and may either converge to the horizon or converge to space infinity, depending whether the initial velocity is below or above the escape velocity function $\sqrt{2M/r}$. 

\item Hence, we may need more than three pieces for the fluid paths, and the optimal path may contain countably many pieces, each being a part of a characteristic curve or a part of the boundary.  

\item Clearly, in the limit of vanishing mass $M \to 0$, our formula reduces to the one known for Burgers equation with a boundary \cite{Joseph}. 

\eei

\noindent We refer to Figures~\ref{Fig2} and \ref{Fig3} for two plots of the characteristics emanating from a point and the velocity achieved along these curves. 

\begin{theorem}[Existence theory for the Burgers--Schwarzschild equation]
\label{theo31}
 Given a radius $r_* >2M$ and an initial data\footnote{Clearly, our assumption $u_0 \in L^1([r_*, +\infty))$ is equivalent to $U_0 \in L^\infty([r_*, +\infty))$, since the factor $1-2m/r$ remains bounded near $r_*$ and tends to $1$ at infinity.}
 $u_0 \in L^1([r_*, +\infty))$ prescribed at some time $t=t_0 \in \RR$ and any boundary data $u_*:[t_0, +\infty) \to (-1,1)$ with flux $f_* \in L^1_\loc([t_0, +\infty))$ (cf. the definition \eqref{eq-fluxstar}), the formula \eqref{eq:variational-principle-for-Schw}--\eqref{eq:flux-on-characteristic} determines a bounded variation solution to the Burgers--Schwarzschild equation, which is defined globally for all $t \geq t_0$ and $r>r_*$.
\end{theorem} 

\begin{proof}
It only remains to prove that the formula above makes sense, that is, to prove that for every base point $(t_1, r_1)$ with $t_1 > t_0$ and $r_1 > r_*$, there exists a (possibly non-unique) minimizer having velocities $\in (-1,1)$. In order to show that this, we consider the limit of large negative velocities, the argument for large positive velocities being similar. So, we assume that $u_1 \to -1$, for which the characteristic curve is given for $t_0 \leq t \leq t_1$ by  
\be
e^{\chi^\sharp(t) - r_1} \Bigg({ \chi^\sharp(t) - 2M \over r_1 - 2M} \Bigg)^{2M} = e^{-(t-t_1)}  \qquad \text{ in the limit } u_1=-1, 
\ee
so that the flux converges to  
$$
F(\chi^\sharp(t)) =  {2M \over \chi^\sharp(t) - 2M} \qquad \text{ in the limit } u_1=-1.
$$
We now analyze the variation of $F$ near this value and show that 
$$
{\del \over \del u_1} \Big( U_0(\chi(t_0)) + \int_{t_0}^{t_1} F(\chi(t)) \, dt \Big) |_{u_1=-1} < 0, 
$$
so that the minimum can only be achieved for velocities {\sl strictly bigger} than $-1$. Namely, with $r_0 = \chi(t_0)$, we have 
\bel{eq:each1}
{\del \over \del u_1} \Big( U_0(\chi(t_0)) \Big) = \Big(1 - {2M \over r_0} \Big)^{-2} u_0(r_0) {\del r_0 \over \del u_1}, 
\ee
while 
\bel{eq:each2}
\aligned
{\del \over \del u_1} \Big( \int_{t_0}^{t_1} F(\chi(t)) \, dt \Big)
& =
\int_{t_0}^{t_1} {\del \over \del u_1} 
\Bigg( 
{1 \over 2} \Bigg( {u_1^2 - 1 \over 1 - {2M \over r_1}} +1 \Bigg) 
   + {2M \over \chi(t) - 2M}
\Bigg) \, dt
\\
& =
\int_{t_0}^{t_1} 
\Bigg( {u_1 \over 1 - {2M \over r_1}} 
   - {2M \over (\chi(t) - 2M)^2} {\del \chi(t) \over \del u_1} 
\Bigg) \, dt. 
\endaligned
\ee

On the other hand, from the differential equation defining $\chi(t)$ we find for the derivative $\xi^\sharp(t) := {\del \over \del u_1} \chi(t)|_{u_1=-1}$: 
$$
\xi'(t)  = - {2M \over \chi^\sharp(t)^2} \xi(t)
            - {\Big( 1 - {2M \over \chi^\sharp(t)} \Big)^2 \over 1 - {2M \over r_1}}.
$$
Since $\xi(t_1)=0$, we obtain 
$$
\xi(t)  = 
            \int ^{t_1}_t e^{\int_t^{t'} {2M \over \chi^\sharp(t)^2}  dt} 
{\Big( 1 - {2M \over \chi^\sharp(t')} \Big)^2 \over 1 - {2M \over r_1}} dt' >0. 
$$
Therefore, each term in \eqref{eq:each1} and \eqref{eq:each2} is negative and we have the desired conclusion, that is, the existence of a minimizer with velocities $\in (-1,1)$. 
\end{proof}

\

\begin{figure}[htbp]
\begin{centering}
\includegraphics[width=8cm,height=8cm,angle=0]{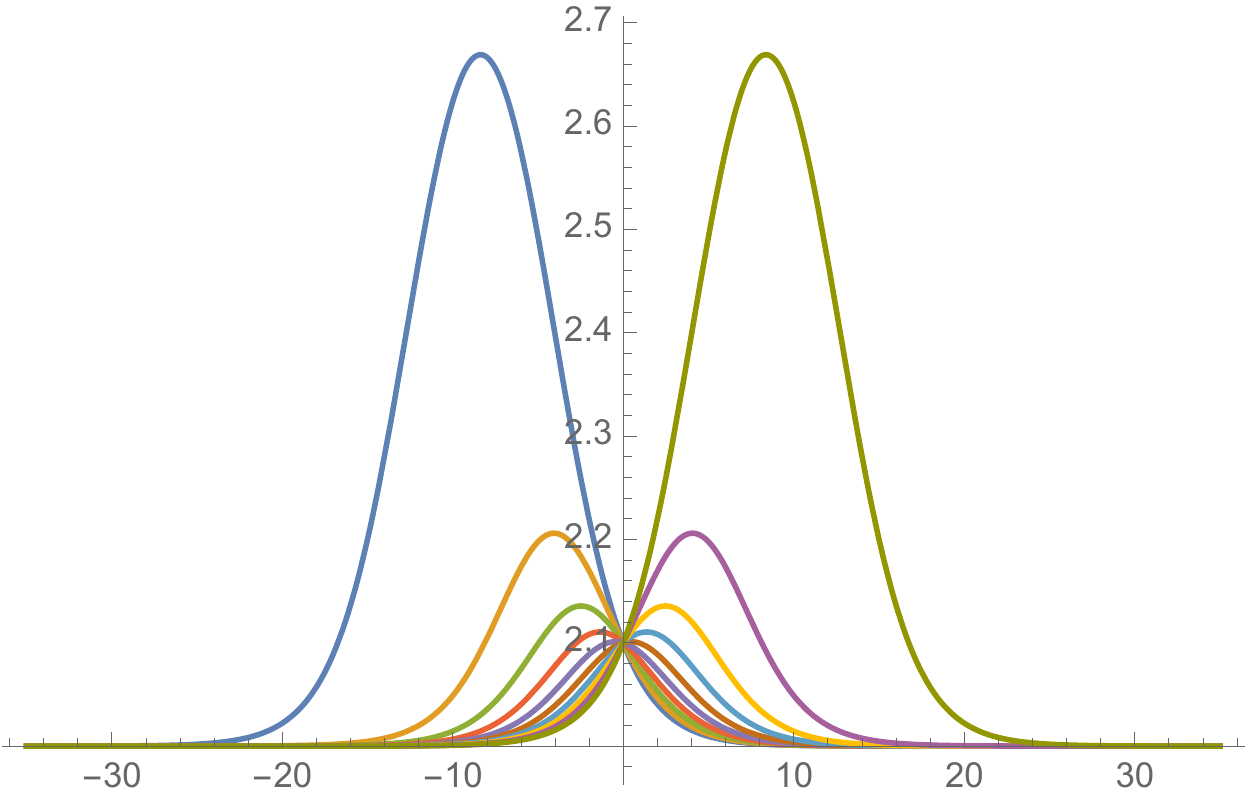} 
\includegraphics[width=8cm,height=8cm,angle=0]{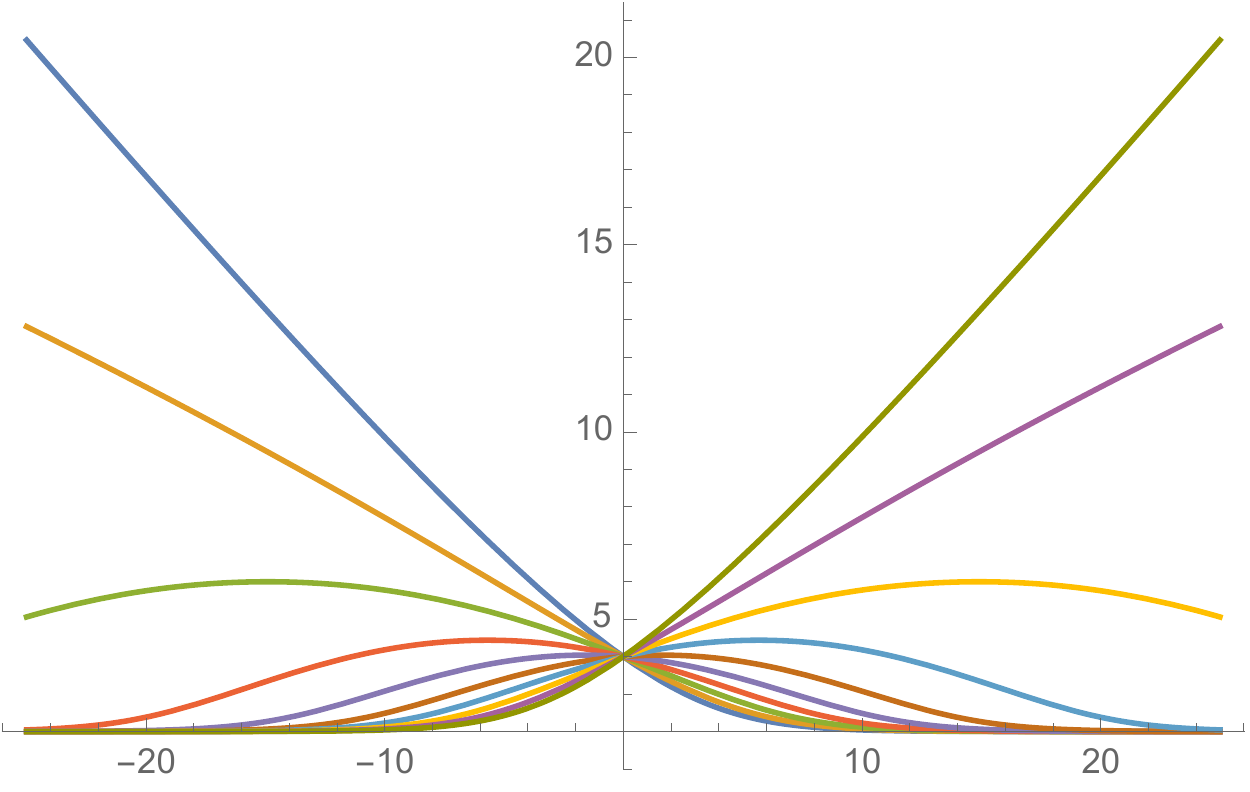} 
\par\end{centering}
\caption{Two plots of the future and past characteristic curves emanating from a point $r_0$}
\label{Fig2} 
\end{figure}

\begin{figure}[htbp]
\begin{centering}
\includegraphics[width=8cm,height=8cm,angle=0]{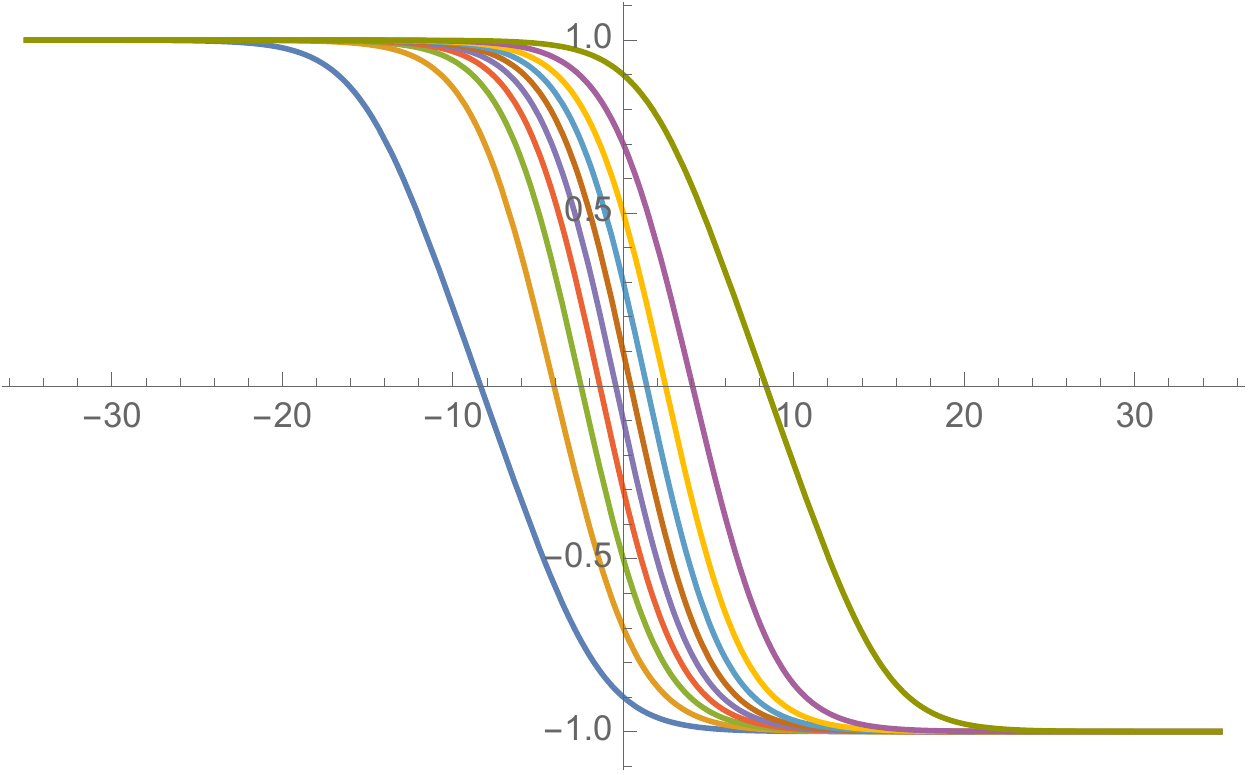} 
\includegraphics[width=8cm,height=8cm,angle=0]{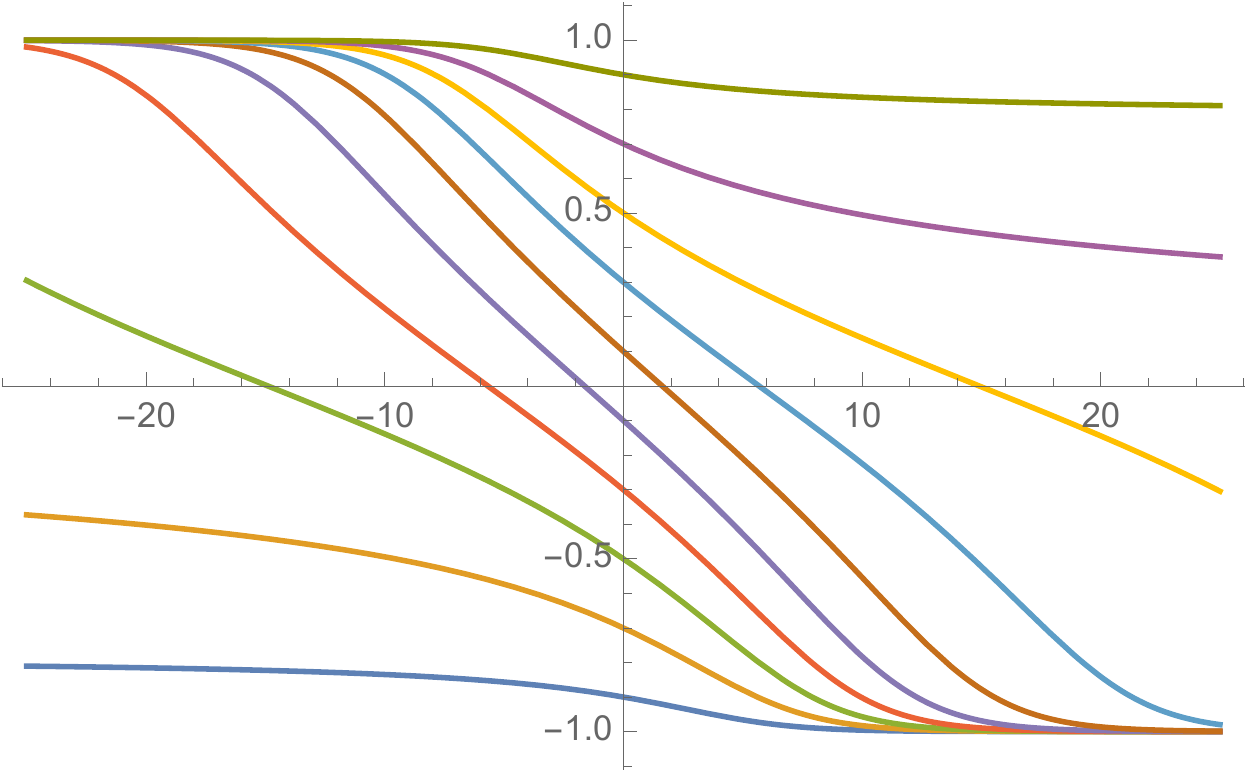} 
\par\end{centering}
\caption{Two plots of the velocity along future and past characteristic curves emanating from a point $r_0$}
\label{Fig3} 
\end{figure}

%--------------------------------------------------------------------------------------------------------------------------------------- 

\subsection{Existence theory for the Euler--Schwarzschild model} 
\label{sec35}

We now rely on the explicit formula from the previous section and analyze the global existence problem for the full model of interest. The following issues remain to be discussed: 

\bei 

\item introducing a notion of weak solution for the integrated density equation, 

\item solving the density equation, and 

\item dealing with initial data for which the velocity field need not be integrable at infinity. 

\eei 

\noindent Following LeFloch \cite{LeFloch90} (who treated the case $M=0$) and relying on Volpert's notion of superposition \cite{VOLP}, we introduce the following concept of weak solution. Given two functions of locally bounded variation in space, say $u=u(t,r)$ and $v=v(t,r)$, we define the product $\uhat \del_r v$ 
by selecting a pointwise representative $\uhat$ of the function $u$, that is, 
\be
\uhat (t,r) = \begin{cases}
u(t,r) \quad & \text{ if $r$ is a point of continuity of $u(t, \cdot)$}, 
\\
{1 \over 2} \big( u_-(t,r) + u_+(t,r) \big), \quad & \text{ if $r$ is a jump point with traces $u_\pm(t, r)$}.
\end{cases} 
\ee

\begin{definition} Given $M>0$ and some radius $r_* \in (2M, +\infty)$, a {\em weak solution to Euler--Schwarzschild model} is a pair of functions $u=u(t,r) \in (-1, 1)$ and $v=v(t,r) \geq 0$ with locally bounded variation in space, defined for all $t \geq 0$ and $r >r_*$ and satisfying 
\be
\label{eq:Eu-Schw}
\aligned
& \del_t u + \Big( 1 - {2M \over r} \Big) \del_r \big( {u^2 \over 2} \Big) 
= {M \over r^2} ( u^2 - 1), 
\\
&\del_t v + \big( 1- {2M \over r} \big) \uhat \del_r v = 0. 
\endaligned
\ee 
\end{definition}

Based on our analysis in the previous section, we now arrive at our main existence results.  

\begin{definition} A function $u_0=u_0(r)$ defined for all $r>r_*$ is said to have the {\sl asymptotic velocity} $u_\infty^\sharp \in (-1, 1)$ if  the function $r \mapsto (u_0(r) - u_\infty^\sharp)$ is integrable at infinity.  
\end{definition}

\begin{theorem}[Existence theory for Euler--Schwarzschild system --- data with prescribed asymptotic velocity] Fix a black hole mass $M>0$, a time $t_0 \in \RR$, and a radius $r_* \in (2M, +\infty)$. Given any initial velocity $u_0: [r_*, +\infty) \mapsto (-1, 1)$ with asymptotic velocity $u_\infty^\sharp \in (-1, 1)$ and any initial integrated density $v_0: [r_*, +\infty) \mapsto [0, +\infty)$ with locally bounded variation,
and given any boundary velocity $u_*: [t_0, +\infty) \mapsto (-1,1)$  with flux $f_* \in L^1_\loc([t_0, +\infty))$  
and any boundary integrated density $v_*: [t_0, +\infty) \mapsto [0, +\infty)$ with locally bounded variation, 
the initial value problem associated with the Euler--Schwarzschild model \eqref{eq:ES} admits a global solution  $u=u(t,r) \in (-1, 1)$ and $v=v(t,r) \geq 0$ defined for all $t \geq t_0$ and $r > r_*$, which have locally bounded variation in space for $t >t_0$ and such that the component $u$ has asymptotic velocity $u_\infty^\sharp$ for all times $t \geq t_0$. 
\end{theorem}

\begin{proof} 
Observe that in the above theorem the initial data is not assumed to be integrable, so that our previous definition 
\eqref{eq:defU} need not make sense. However, we can easily redefine $U$ by substracting the static solution  $u^\sharp=  u^\sharp(r)$ associated with its asymptotic velocity $u^\sharp_\infty$, that is, 
\bel{eq:defU2}
U(t,r) = - \int_r^{+\infty} \Big(1 - {2M \over r'} \Big)^{-2} \big( u(t,r') - u^\sharp(r') \big)  \, dr', \qquad r \geq r_*.
\ee  
Then, by repeating the same calculation as in the previous section, we obtain exactly the same inequalities and therefore the same explicit formula but in terms of the new function \eqref{eq:defU2}. 

Once the solution component $u=u(t,r)$ is determined by the explicit formula, we obtain the solution component $v=v(t,r)$ by setting
\be
v(t,r) = \begin{cases} 
v_0(\chi(0; t, r)), \quad \text{ if the optimal path reaches the initial line,}
\\
v_*(s_* +), \quad \text{ if the optimal path reaches the boundary at some time $s_*$,}
\end{cases} 
\ee
in which $s \mapsto \chi(s; t,r)$ denotes the backward minimizing path from $(t,r)$. Observe that the path may reach a point of discontinuity of the initial or boundary data and, for definiteness, we select the right-hand representative of the BV function $v_*$.  
\end{proof}

Finally, we can cover initial data for which no restriction is imposed at infinity. This is achieved by redefining again the function $U$ by integrating {\em from the boundary,} i.e. 
\bel{eq:defU3}
U(t,r) = \int_{r_*}^r \Big(1 - {2M \over r'} \Big)^{-2} u(t,r') \, dr', \qquad r \geq r_*. 
\ee  
This function significantly differs from the original definition by a {\em flux term,} which is the boundary flux actually achieved by the solution and need not coincide with the prescribed boundary data. The boundary flux achieved by a solution is not a priori known and we refer to LeFloch \cite{MR949657} for a characterization of this term. However, we do not need it here in order to express the minimization property of interest, and it is not difficult to repeat our previous calculations and re-derive all of our inequalities by including this term throughout the inequalities. We arrive at the following conclusion.

\begin{theorem}[Existence theory Euler--Schwarzschild system --- non-integrable data]
Fix a black hole mass $M>0$, a time $t_0 \in \RR$, and a radius $r_* \in (2M, +\infty)$. Given with any initial velocity $u_0: [r_*, +\infty) \mapsto (-1, 1)$ in $L^1_\loc$
 and any initial integrated density $v_0: [r_*, +\infty) \mapsto  [0, +\infty)$ 
with locally bounded variation, 
and given any boundary velocity $u_*: [t_0, +\infty) \mapsto (-1,1)$  with flux $f_* \in L^1_\loc([t_0, +\infty))$  
and any boundary integrated density $v_*: [t_0, +\infty) \mapsto [0, +\infty)$ with locally bounded variation, 
the initial value problem associated with the Euler--Schwarzschild model admits a global solution 
admits a global solution  $u=u(t,r) \in (-1, 1)$ and $v=v(t,r) \geq 0$ defined for all $t \geq t_0$ and $r > r_*$, which have locally bounded variation in space for $t >t_0$ and such that the component $u$ has asymptotic velocity $u_\infty^\sharp$ for all times $t \geq t_0$. 
\end{theorem}

%==================================================================

\section{Burgers equation with random forcing on a boundary} 
\label{sec5}

\subsection{The Hopf--Lax--Oleinik formula with a boundary}

The framework we are going to describe in this section has common features with \cite{B:MR2299503,B:MR3112935}. We assume here that the black hole mass $M=0$ vanishes, so that our relativistic Burgers equation reduces to the standard Burgers equation.  Hence, letting $M \to 0$ in \eqref{eq:ES}, we thus consider the pressureless Euler system on the half-line $\RR_+=[0,+\infty)$:
\begin{align}
\label{eq:Burgers}
\partial_t u + \partial_x\left(\frac{u^2}{2}\right)&=0,
\\
\label{eq:scalar-in-Burgers-field}
 \partial_t v+u\partial_x v&=0.
\end{align}
Here, $u=u(t,x)$ and $v=v(t,x)$ are functions of time $t\in\RR$ and space variable $x\in\RR_+$. We are going to equip this system with stationary random boundary conditions and 
study the long-term statistical properties of solutions. Since \eqref{eq:Burgers}--\eqref{eq:scalar-in-Burgers-field} is a system of hyperbolic equations, it is standard to work
with entropy solutions.

A unique entropy solution of the initial-boundary value problem associated with the equations \eqref{eq:Burgers}--\eqref{eq:scalar-in-Burgers-field}, initial conditions
\begin{align}
\label{eq:initial-u}
 u(t_0,x)&=u_0(x),\quad x>0,\\
\label{eq:initial-v}
 v(t_0,x)&=v_0(x),\quad x>0,
\end{align}
and boundary conditions
\begin{align}
 \label{eq:boundary-u}
 u(t,0)=\phi(t), \quad t\ge t_0,\\
 \label{eq:boundary-v}
 v(t,0)=\psi(t), \quad t\ge t_0,
\end{align}
is given by the following two-stage procedure. 

1. First of all, we consider the following minimization problem f(or every $t>t_0$ and $x>0$):
\begin{equation}
\label{eq:basic-var-problem-for-Burgers}
U(t,x)=\inf_{\gamma: \gamma(t)=x}\left\{U_0(\gamma(t_0))+\frac{1}{2}\int_{t_0}^t \dot\gamma^2(s)\ONE_{\gamma(s)>0}ds - \frac{1}{2}\int_{t_0}^t \phi_+^2(s)\ONE_{\gamma(s)=0}ds\right\}.
\end{equation}
Here, the infimum is taken over all absolutely continuous curves $\gamma:[t_0,t]\to\RR_+$ satisfying $\gamma(t)=x$, and $U_0$ is defined by
\begin{equation}
\label{eq:initial-condition-potential}
 U_0(y)=\int_0^y u_0(z)dz,\quad y\in\RR_+.
\end{equation}
For $p\in\R\cup\{-\infty\}$, we denote 
\[
\U_p=\left\{u\in L^1_\loc(\R_+):\ \liminf_{x\to + \infty}\frac{\int_0^x u(y)dy}{x} > p\right\}. 
\]
Assuming that $\phi_+$ is a function that is Lebesgue integrable on bounded intervals and that $u_0\in\U=\U_{-\infty}$,
% is a measurable function such that
%\begin{equation}
%\label{eq:finite-flux-from-infinity}
%\liminf_{x\to+\infty}\frac{U_0(x)}{x}>-\infty,
%\end{equation}
the value $U(t,x)\in\RR$ is well-defined by~\eqref{eq:basic-var-problem-for-Burgers} for each $t>t_0$ and $x>0$. For every fixed $t>t_0$, there is an at most countable set $B(t)\subset\RR_+$ such that for every $x\in\RR_+\setminus B(t)$,
the minimal action $U(t,x)$ is attained at a unique minimizer $\gamma^*_{(t,x)}$. Moreover, for every $x\in\RR_+\setminus B(t)$, the derivative $\partial_x U(t,x)$ is well-defined and equals $\dot\gamma^*_{(t,x)}(t)$. At these points, the unique entropy solution of the Burgers equation is given by 
\be
u(t,x)=\partial_x U(t,x).
\ee
 We also
define the solution at shock points $x\in B(t)$ to be right-continuous by setting $u(t,x)=\lim_{y\downarrow x}u(t,y)$. In fact, it is easy to see that thus defined solutions are c\`adl\`ag (i.e. right-continuous with left limits; we denote the set of all such functions by $\D$), and if $u(t_0,x)\in \U_p$ for some $p$, then so is $u(t,\cdot)$ for all $t>t_0$.

\vskip.3cm

2. In a second stage of the construction, we observe that the action minimizers $\gamma$, in fact, have very simple structure. Let us denote $S=\{s\in[0,t]:\ \gamma(s)=0\}$. If $S=\emptyset$, then $\dot\gamma$ is constant. If $S\ne\emptyset$, then~$S$ is a segment $[a,b]$ for some $a,b\in[t_0,t]$,   
and $\dot\gamma$ is constant on $[t_0,a)$ and $(b,t]$. In the former case, $\gamma$ is a line segment in space-time. In the latter case, $\gamma$ is a broken line (i.e. a piecewise linear path). In both cases, the last straight segment of $\gamma$ is a characteristic of
~\eqref{eq:Burgers}. Given $u$, the characteristics of the equation on $v$ coincide with constant velocity characteristics for~\eqref{eq:Burgers}, so we can define $v(t,x)$ to coincide with $v_0(\gamma(t_0))$ in the first case, and with $\psi(b)$ in the second case.

This procedure defines $u(t,x)$ and $v(t,x)$ uniquely for each $t>t_0$ and $x\in\RR_+\setminus B(t)$. At points of $B(t)$, the solutions have discontinuities.

%----------------------------------------------------------------------------------------------------------------------------

 \subsection{Ergocity properties for the Burgers equation with a boundary}

We are interested in the situation where the boundary conditions $(\phi(t),\psi(t))_{t\in\RR}$ are random and form a stationary stochastic process. Namely, we will assume that on some probability space $(\Omega,\Fc,\Pp)$ there is an ergodic flow $(\theta^t)_{t\in\RR}$ of
$\Pp$-preserving measurable transformations of $\Omega$, and two random variables $\phi:\Omega\to\RR$ and $\psi:\Omega\to\RR$ such that for all $\omega\in\Omega$, the stationary processes $\phi(t,\omega)=\phi(\theta^t\omega)$ and $\psi(t,\omega)=\psi(\theta^t\omega)$ are bounded on any finite time interval.  We assume that $q^2 := \E \phi_+^2<+\infty$.
Our goal is, for $\Pp$-a.e.\ $\omega$, i.e. for a.e.\ realization of $(\phi,\psi)$,  

\bei 

\item[(i)] to construct a global solution
$(u,v)$ of the system~\eqref{eq:Burgers}--\eqref{eq:scalar-in-Burgers-field} with boundary conditions 
$(\phi(t),\psi(t))$ given for all times $t\in\RR$; 

\item[(ii)] to study some properties of the global solution;

\item[(iii)] and to describe the basin of pullback attraction of the global solution and its uniqueness class.

\eei 

To state our main results on the Burgers equation, we introduce the family of random evolution operators
\be
 \Phi_{\omega}^{t_0,t_1}:\U\to\U\cap \D,\quad \omega\in\Omega,\ t_0<t_1,
\ee
such that $\Phi^{t_0,t_1}_\omega u$ denotes 
 the solution at time $t_1\ge t_0$ of the Cauchy problem for the Burgers equation with initial condition $u\in \U$ at time $t_0$ and boundary condition $\phi=\phi(\cdot,\omega)$. This family of operators $\Phi$ forms a cocycle, i.e. it satisfies the following identity:
\be
 \Phi^{t_0,t_2}_\omega =\Phi^{t_1,t_2}_\omega\circ\Phi^{t_0,t_1}_{\omega},\quad t_0\le t_1\le t_2.
\ee
Equivalently, one could work with the family of operators $\Phi^t_\omega=\Phi^{0,t}_\omega$, and in terms of these operator, the cocycle property reads 
\be
 \Phi^{t_1+t_2}_\omega=\Phi^{t_2}_{\theta^{t_1}\omega}\circ \Phi^{t_1}_\omega,\quad t_1,t_2\ge 0.
\ee
Under our assumptions on the random boundary condition $\phi$, these cocycle identities hold for all~$\omega$, and there is no need to discuss exceptional sets (i.e. this is a perfect cocycle).

It turns out that these families of operators have strong contraction properties that can also be viewed as fast loss of memory in the system. To make this precise, 
similarly to~\cite{B:MR3112935},
we need to introduce a  measure of proximity in the space $\U$. 
For any
$h_1,h_2\in\U$ we denote by $r(h_1,h_2)\in[0,\infty]$ the maximal value $r>0$ such that the restrictions of 
$h_1$ and $h_2$ on $[0,r]$ coincide, and then we define
\begin{equation}
\label{eq:metric-d}
 d(h_1,h_2)=\exp(-r(h_1,h_2)).
\end{equation}
If there is no neighborhood of the origin where $h_1$ and~$h_2$ coincide, we set $d(h_1,h_2)=1$.
If $h_1\equiv h_2$, we set $d(h_1,h_2)=0$.
Thus defined~$d$ is an ultrametric in $\U$ taking values in $[0,1]$. Convergence in this metric simply means that the domain of disagreement of functions ``shrinks to the empty set at infinity''.

\begin{definition}
A $\U$-valued process $u_\omega(t,\cdot)$ (with $t\in\R$) is called a {\sl global solution} if for $\Pp$-a.e.~$\omega$ and all $t_0,t_1\in\R$ satisfying $t_0<t_1$, 
\[
\Phi_\omega^{t_0,t_1}u_\omega(t_0,\cdot)= u_\omega(t_1,\cdot).
\]
\end{definition} 

We are now ready to state our main result on stationary solutions of the Burgers equation with random boundary condition (without taking into account the passive scalar field yet; see Section~\ref{subsec44}).

\begin{theorem}\label{thm:1F1S-Burgers} 
Consider the Burgers equation in a half space $[0, +\infty)$ with random boundary conditions described above.
\begin{enumerate}
\item\label{item:existence} There exists a stationary global solution $u_\omega$ satisfying
$u_\omega(t,\cdot)=u_{\theta^t\omega}(\cdot)$. 
\item\label{item:uniqueness} The global solution $u$ constructed in Part~\ref{item:existence} is unique
in the following sense: let $\tilde u$ be a stationary global solution of the Burgers equation~\eqref{eq:Burgers} with boundary condition~\eqref{eq:boundary-u} such that with probability $1$, $\tilde u(t,\cdot)\in \U_{-q}$ for all $t$. Then, with probability 1, $\tilde u$ coincides with the solution $u$ constructed in Part~\ref{item:existence}.
\item\label{item:pullback}  The solution $u$ plays the role of a one-point random attractor on $\U_{-q}$. Namely, if $w\in\U_{-q}$,
then the following pullback attraction holds with probability~1:
\begin{equation}
\label{eq:pull_attr}
 d(\Phi^{t,0}_{\omega} w, u_\omega)\to 0,\quad t\to-\infty.
\end{equation}

\end{enumerate} 
\end{theorem}

The conditions involving $\U_{-q}$ in Parts~\ref{item:uniqueness} and~\ref{item:pullback} mean 
that the flux of particles coming from infinity is not strong enough
compared to the average flux of particles generated by the boundary condition. 
In fact, Parts~\ref{item:uniqueness} and~\ref{item:pullback} show that in the long term, the 
initial condition gets forgotten by the system, the long-term behavior of the solution is structured only through the influence of the boundary conditions, and the only velocity profile at a time $t$ compatible with the history of the force up to $t$ is, almost surely, given by $u_\omega(t,\cdot)$. Such a statement is often referred to as a {\sl one-force-one-solution principle} (1F1S in short) and is also known under the name of {\sl synchronization by noise}.

If we introduce additional assumptions on the rate of
convergence of the ergodic averages of $\phi_+^2$ to their mean value
$q^2 = \E \phi_+^2$, then we can
make a claim about the limiting behavior of $u_\omega(t,x)$ as $x\to + \infty$. 
Namely, let us require that 
\begin{equation}
\label{eq:large-deviation-assumption}
\Pp(A(t,\eps)) < \alpha(t,\eps), 
\end{equation}
where
\[A(t,\eps)=\left\{\exists s>t:\quad  \left|\frac{1}{s}\int_{0}^{s}\phi_+^2(r)dr - q^2 \right|\ge \eps\quad \textrm{or}\quad \left|\frac{1}{s}\int_{-s}^{0}\phi_+^2(r)dr - q^2 \right|\ge \eps\right\},\]
and $\alpha$ is some function satisfying, for every $r>0$ and every $\eps>0$, 
\begin{equation}
  \label{eq:mixing-rate}
  \sum_{n=1}^{\infty} \alpha(nr,\eps) <+\infty.
 \end{equation}

\begin{theorem}\label{thm:u(plus-infty)}
Under assumptions \eqref{eq:large-deviation-assumption}--\eqref{eq:mixing-rate},
the global solution $u_\omega$ constructed in Theorem~\ref{thm:1F1S-Burgers} satisfies, with probability $1$, 
\begin{equation}
\label{eq:limit-u-is-q}
 \lim_{x\to + \infty}  u_\omega(t,x)=q,\qquad t\in\R. 
\end{equation}
In particular, $u_\omega(t,\cdot)\in \U_p$ for all $p<q$ and all $t \in \RR$.
\end{theorem}

It is also possible to use the averaging rates for~$\phi^2_+$ to establish  estimates on the rates of convergence in Theorem~\ref{thm:1F1S-Burgers}.

%-----------------------------------------------------------------------------------------------------------------------------------

\subsection{Proofs of Theorems~\ref{thm:1F1S-Burgers} and~\ref{thm:u(plus-infty)}.}

\begin{proof}[Proof of Theorem~\ref{thm:1F1S-Burgers}]
1. 
First, we construct a global stationary solution. For every point $(t,x)\in(0,+\infty)\times \RR$, we can consider the following
variational problem (obtained by essentially removing the contribution from the initial data): 
\begin{equation}
 \label{eq:constructing-global-solution}
 F(t,x,s)=\frac{x^2}{2(t-s)}+\frac{1}{2}\int_{s}^{t}\phi^2_+(r)dr \to \inf,\quad s<t.
\end{equation}
Since $F(t,x,\cdot)$ is a continuous function on $(-\infty,t)$ such that $F(t,x, t-)=F(t,x,-\infty)=+\infty$,
the minimum value in~\eqref{eq:constructing-global-solution} is attained, and all points providing the  minimum belong to a bounded interval $[t_0,t]$. 

Comparing~\eqref{eq:constructing-global-solution} to a variational problem
\[
 \frac{x^2}{2(t-s)}-\frac{1}{2}\int_{t_0}^{s}\phi^2_+(s)ds \to \inf, \qquad s\in[t_0,t], 
\]
and recalling that 
\[
\frac{x^2}{2(t-s)}=\inf_{\gamma}\frac{1}{2}\int_s^t \dot\gamma^2(r)dr= \frac{1}{2}\int_s^t \dot\gamma_*^2(r)dr,
\]
 where the infimum is taking over all paths connecting $0$ to $x$ over the time interval from $s$ to $t$,
and~$\gamma_*$ is the only path among those with constant velocity $x/(t-s)$, we obtain that the variational
problem~\eqref{eq:constructing-global-solution} gives rise to a global solution of the Burgers equation.
Namely, let $t_*=t_*(t,x)$ denote the smallest of the points where the minimum of $F(t,x,\cdot)$
is attained. Defining
\begin{equation}
 \label{eq:global-solution}
  u(t,x)=\frac{x}{t-t_*(t,x)},\quad (t,x)\in \RR\times (0,+\infty),
\end{equation}
and using that subpaths of minimizers are also minimizers, we obtain that $u$ 
is a global stationary solution of the Burgers equation~\eqref{eq:Burgers} with boundary condition~$\phi$,
c\`adl\`ag for every $t$. This completes the proof of Part~\ref{item:existence} of the theorem.

%For every point $(t,x) \in \RR\times(0,\infty)$, we denote by $I_{(t,x)}=[(t_*(t,x),0), (t,x)]$ the straight line characteristic segment connecting $(t,x)$ to  $\RR\times \{0\}$. 

%\begin{lemma}\label{lem:chars-do-not-intersect} Characteristics ending at any distinct points $(t_1,x_1), (t_2,x_2)\in \RR\times(0,\infty)$, do not intersect, i.e. we have $I_{(t_1,x_1)}\cap I_{(t_2,x_2)}=\emptyset$.\end{lemma}

\smallskip

2. Let us now establish Part~\ref{item:pullback}. We need to prove the following statement: for every~$w$ satisfying the conditions of the theorem and every $r>0$,
there is $t_0<0$ such that for all $t<t_0$,
the restrictions of $u(0,\cdot)$ and $\Phi_\omega^{t,0}w(\cdot)$ onto $[0,r]$ coincide.

We will begin with the claim that there is $t_1<0$ such that for all $t<t_1$, the minimizers in the variational principle~\eqref{eq:basic-var-problem-for-Burgers}
constructed for initial condition $w$ given at time $t$ terminating at point $(0,x)$ with $x\in[0,r]$ cannot consist of one segment with starting point at $(t,y)$
for some $y\in[0,+\infty)$. To see this, we observe that our condition on $w$ implies that there is $y_0>0$ and $q'\in (0,q)$ such that $W(y)\ge -q'y$ for $y\ge y_0$, where $W(y)=\int_0^yw(z)dz$. Also
there is a constant $c>0$ such that $|W(y)|\le c/2$ for $y\in[0,y_0]$.

For any point $(0,x)$ with $x\in(0,r]$, let us compare the action $A$ of the motion with constant velocity
between points $(t,y)$ and $(0,x)$ to the action $B$ of the path that stays at $0$ on the time interval 
$[t,-r]$ and travels to $x$ between times~$-r$ and~$0$ with constant velocity $x/r$.

We have
\[
 A=\frac{(x-y)^2}{2|t|}+W(y)\ge -c/2,\quad y<y_0,
\]
and 
\begin{align*}
A=&\frac{(x-y)^2}{2|t|}+W(y)\ge \frac{(x-y)^2}{2|t|}-q'y
= \frac{x^2+y^2-2xy}{2|t|}-q'y\\
  \ge& \frac{x^2}{2|t|}+\frac{y^2-2(|t|q'+x)y+(|t|{q'}+x)^2}{2|t|}-\frac{(|t|q'+x)^2}{2|t|}
  \\ \ge& -\frac{(q'+r/|t|)^2|t|}{2},\quad y\ge y_0, 
\end{align*}
while
\[
 B= \frac{x^2}{2r}-\frac{1}{2}\int_{t}^{-r} \phi_+^2(s)ds\le \frac{r}{2}-\frac{1}{2}\int_{t}^{-r} \phi_+^2(s)ds.
\]
If the minimizer is a constant velocity path starting at $(t,y)$ for some $y>0$, we obtain
\[
 \frac{r}{2}-\frac{1}{2}\int_{t}^{-r} \phi_+^2(s)ds\ge -\frac{(q'+r/|t|)^2|t|\vee c}{2},
\]
i.e.
\[
 \frac{1}{-r-t}\int_{t}^{-r} \phi_+^2(s)ds\le  \frac{(q'+r/|t|)^2|t| \vee c+ r}{-r-t}\le q^2-\delta,
\]
for some $\delta>0$ if $|t|$ is sufficiently large. Applying the ergodic theorem to
the process $\phi$,  we obtain that, with probability 1, this situation cannot happen for
sufficiently large $|t|$. Therefore, for all $t$ less than some $t_1$, the minimizers have to come from the boundary.

It is also clear that if $t< t_2:=\inf \{t_*(0,x):\ x\in(0,r]\}=t_*(0,r)$, then the minimizers for $(0,x)$, $x\in(0,r]$ have to depart from the boundary at $t_*(0,x)$, so, taking $t_0=t_1\wedge t_2$, and we have completed the proof of Part~\ref{item:pullback} of the theorem.  

\medskip

3. It remains to prove Part~\ref{item:uniqueness} of Theorem~\ref{thm:1F1S-Burgers}. It is sufficient to check that two solutions coincide on a dense set of space-time points.
By stationarity, it suffices  to check that for any $x_0>0$, $\tilde u(0,x_0)=u(0,x_0)$ with probability~$1$.

We denote $\tilde U(t,x)=\int_0^x \tilde u(t,y) dy$.
Due to stationarity of $\tilde u$, we can find $y_0>0$, $c>0$,  $q'\in (0,q)$,
and a sequence of times $t_k\to-\infty$ such that 
for all $k$,  $\tilde U(t_k, y)\ge -q'y$ for $y\ge y_0$ and 
$|\tilde U(t_k,y)|\le c/2$ for $y\in[0,y_0]$. Repeating the argument from the proof of 
Part~\ref{item:pullback}, we obtain that, with probability~1, for sufficiently large $t_k$,
\[
\tilde u(0,x)=(\Phi^{t_k,0}_\omega \tilde u(t_k,\cdot))(x)=\frac{x}{-t_*(0,x)}=u(0,x),
\]
which completes the proof of Part~\ref{item:uniqueness} and, therefore, the whole proof of Theorem~\ref{thm:1F1S-Burgers}. 
\end{proof}

\begin{proof}[Proof of Theorem~\ref{thm:u(plus-infty)}]
Since the characteristics used in the construction of the global solution~$u$ do not intersect themselves inside $\R_+\times(0,\infty)$, it is sufficient to prove that with probability~1, the limiting relation~\eqref{eq:limit-u-is-q} holds for all $t\in\Z$, with the limit understood as
$\N\ni x \to + \infty$. Due to the stationarity in~$t$ and countable additivity of the probability measure, all we need to prove is 
\begin{equation}
\label{eq:limit-u-is-q-natural}
 \lim_{\N\ni x\to + \infty}  u(0,x)=q,
\end{equation}
i.e. 
for any $\eps>0$, with probability~$1$,
\begin{equation}
\label{eq:rephrasing-limit}
 -\eps <\frac{x}{-t_*}-q<\eps 
\end{equation} 
for all sufficiently large~$x\in\N$.  
Let $ t_q=t_q(x)=-x/q$.
We have
\begin{align}
\notag
0&\ge F(0,x,t_*(0,x))-F(0,x,t_q(x))=\frac{x^2}{-2t_*}-\frac{x^2}{-2t_q}-\frac{1}{2}\int_{t_q}^{t_*}\phi_+^2(r)dr
\\ &=\frac{x^2}{2}\frac{t_*-t_q}{t_*t_q}-\frac{1}{2}\int_{t_q}^{t_*}\phi_+^2(r)dr =\frac{xq(t_*-t_q)}{-2t_*}-\frac{1}{2}\int_{t_q}^{t_*}\phi_+^2(r)dr. 
\label{eq:diff-actions}
\end{align}
Here, $F$ was introduced in~\eqref{eq:constructing-global-solution}, and by convention we always have 
$\int_a^b=-\int_b^a$.
Suppose that the second inequality in~\eqref{eq:rephrasing-limit} is violated. Then, on the one hand,
\begin{equation}
\label{eq:lower-est-on-time-increment-1}
 t_*-t_q\ge-\frac{x}{q+\eps}+\frac{x}{q}=x\frac{\eps}{(q+\eps)q}, 
\end{equation}
and on the other hand,~\eqref{eq:diff-actions} can be continued as
\begin{equation*}
%\label{eq:linear-growth-of-difference-actions1}
0\ge \frac{1}{2}\left(q(q+\eps)(t_*-t_q)-\int_{t_q}^{t_*}\phi_+^2(r)dr\right),
\end{equation*}
or, equivalently,
\begin{equation}
\label{eq:linear-growth-of-difference-actions11}
\frac{1}{t_*-t_q}\int_{t_q}^{t_*}\phi_+^2(r)dr-q^2\ge q\eps.
\end{equation}
Inequalities~\eqref{eq:lower-est-on-time-increment-1} and~\eqref{eq:linear-growth-of-difference-actions11} together with stationarity of $\phi$ and 
the assumption~\eqref{eq:large-deviation-assumption}
imply that 
\begin{equation}
 \Pp\left\{\frac{x}{-t_*}\ge q+\eps\right\}<\alpha\left(x\frac{\eps}{(q+\eps)q}, q\eps\right).
\end{equation}
Applying assumption~\eqref{eq:mixing-rate} and the Borel--Cantelli lemma, we obtain that, with probability~1, there are only finitely many $x\in\N$ such that the second
inequality in~\eqref{eq:rephrasing-limit} is violated.  In other words, that inequality holds for sufficiently large~$x\in\N$.

Suppose now that the first inequality in~\eqref{eq:rephrasing-limit} is violated. Then
\begin{equation}
\label{eq:lower-est-on-time-increment-2}
 t_q-t_*\ge-\frac{x}{q}+\frac{x}{q-\eps}=x\frac{\eps}{(q-\eps)q}, 
\end{equation}
and~\eqref{eq:diff-actions} can be continued as
\begin{equation*}
%\label{eq:linear-growth-of-difference-actions2}
0\ge \frac{1}{2}\left(-q(q-\eps)(t_q-t_*)+\int_{t_*}^{t_q}\phi_+^2(r)dr\right),
\end{equation*}
or, equivalently,
\begin{equation}
\label{eq:linear-growth-of-difference-actions21}
\frac{1}{t_q-t_*}\int_{t_*}^{t_q}\phi_+^2(r)dr-q^2\le -q\eps.
\end{equation}

Inequalities~\eqref{eq:lower-est-on-time-increment-2} and~\eqref{eq:linear-growth-of-difference-actions21} together with stationarity of $\phi$ and 
the assumption~\eqref{eq:large-deviation-assumption}
imply that 
\begin{equation}
 \Pp\left\{\frac{x}{-t_*}\le q-\eps\right\}<\alpha\left(x\frac{\eps}{(q-\eps)q}, q\eps\right).
\end{equation}
Applying assumption~\eqref{eq:mixing-rate} and the Borel--Cantelli lemma, we obtain that, with probability~1, there are only finitely many $x\in\N$ such that the first
inequality in~\eqref{eq:rephrasing-limit} is violated.  In other words, that inequality holds for sufficiently large~$x\in\N$, which completes the proof of Theorem~\ref{thm:u(plus-infty)}.
\end{proof}

%--------------------------------------------------------------------------------------------------------------------

\subsection{Ergodicity properties for the integrated fluid density}
\label{subsec44}

Theorem~\ref{thm:1F1S-Burgers} concerns only the first component of the two-component system~\eqref{eq:Burgers}--\eqref{eq:scalar-in-Burgers-field}.  
As for the second component, we know that $v$ is constant along characteristics, and for every point $(t,x)$ the characteristic in the setting of Theorem~\ref{thm:1F1S-Burgers} is given by a constant velocity motion
from $t_*(t,x)$ to $(t,x)$, if the initial condition is assigned in sufficiently distant past. Combining these facts, we set
\begin{equation}
\label{eq:global_v} 
v(t,x)=\psi(t_*(t,x)),\quad (t,x)\in\R\times(0,+\infty),
\end{equation}
and arrive at the following main conclusion of this section.

\begin{theorem}\label{thm:v} 
Consider the pressureless Euler system in a half space $[0, +\infty)$ with a random boundary condition. 
\begin{enumerate}  
\item (Existence.) Equations~\eqref{eq:global-solution} and \eqref{eq:global_v} define a global stationary solution $(u_\omega, v_\omega)$ of the system~\eqref{eq:Burgers}--\eqref{eq:scalar-in-Burgers-field} with
 boundary conditions~\eqref{eq:boundary-u}--\eqref{eq:boundary-v}.
 \item (Uniqueness.) Every stationary global solution $(\tilde u_\omega, \tilde v_\omega)$ of the system~\eqref{eq:Burgers}--\eqref{eq:scalar-in-Burgers-field} with
 boundary conditions~\eqref{eq:boundary-u}--\eqref{eq:boundary-v} such that $u_\omega(t,\cdot)\in \U_q$ coincides with $(u_\omega, v_\omega)$ with probability~1.
  \item (Pullback attraction.) The solution $(u_\omega, v_\omega)$ is a random one-point pullback attractor, i.e. with probability 1, for every pair $(w,w')$  with $w\in\U_{-q}$ and every $r>0$ there is $t_0<0$
  such that for every $t<t_0$, the solution of the pressureless Euler system with initial condition $(w,w')$ assigned at $t$, coincides with $(u_\omega,v_\omega)$ on $\{0\}\times(0,r]$.
\end{enumerate}
\end{theorem}

%============================================================================

\section{The Euler--Schwarzschild model with random boundary forcing} 
\label{sec6}

\subsection{Ergodicity properties for the Burgers--Schwarzschild equation} 
\label{sec:ergoSch}

Our goal now is to extend the conclusions established in the previous section for the flat
Burgers setting to the Burgers--Schwarzschild setting.  
We begin by introducing the setup and stating our main results concerning the global solutions of the equation satisfied by the velocity, then we will give the proofs of these results in Section~\ref{sec:stationary-Schw-proofs}, and finally, in Section~\ref{sec:Schw-stationary-transport}, we will discuss the transport equation satisfied by the fluid density of the Euler--Schwarschild model.

Similarly to the flat case, we assume that 
 on some probability space $(\Omega,\Fc,\Pp)$ there is an ergodic flow $(\theta^t)_{t\in\RR}$ of
$\Pp$-preserving measurable transformations of $\Omega$, and two random variables 
$\phi:\Omega\to (-1,1)$ and $\psi:\Omega\to\RR$ such that boundary conditions for evolution on 
$[r_*,+\infty)$ for some fixed $r_*>2M$
are given by stationary processes 
\begin{align}
\label{eq:Eu-Schw-boundary-conds}
u_\omega(r_*,t)=\phi(t,\omega)=\phi(\theta^t\omega),\quad
v_\omega(r_*,t)=\psi(t,\omega)=\psi(\theta^t\omega).
\end{align}

First, let us rewrite the variational formula 
\eqref{eq:variational-principle-for-Schw} using more convenient notation for our present purposes. We write
\be
u_E(r) := \sqrt{\frac{2M}{r}},\quad r>2M,
\ee
for the escape velocity at a point~$r$, i.e., the minimum velocity $v$ such that the characteristic curve started at~$r$ with velocity $v$ diverges to infinity. 
The variational formula 
\eqref{eq:variational-principle-for-Schw} can be rewritten as
\begin{equation}
 U(t_1,r_1)=\min_{\chi:\chi(t_1)=r_1} A_{U_0}^{t_0,t_1}(\chi),
\label{eq:rewriting-variational-principle-for-Schw} 
\end{equation}
where we have decomposed the action into several contributions involving the escape velocity and the boundary and initial data 
\begin{equation}
 A_{W}^{t_0,t_1}(\chi)=W(\chi(t_0))+K^{t_0,t_1}(\chi)
+P^{t_0,t_1}(\chi)
 -B^{t_0,t_1}(\chi), 
\end{equation}
with
$$
\aligned
K^{t_0,t_1}(\chi)
& = \frac{1}{2}\int_{t_0}^{t_1}\frac{u^2(t)-u^2_E(\chi(t))}{1-u^2_E(\chi(t))}\ONE_{\chi(t)\in(r_*,\infty]}dt,  
\\
P^{t_0,t_1}(\chi)
& =\int_{t_0}^{t_1}\frac{2M}{\chi(t)-2M}\ONE_{\chi(t)\in(r_*,\infty]}dt,
\\
B^{t_0,t_1}(\chi)
& = \frac{1}{2}\int_{t_0}^{t_1}\frac{\phi^*_+(t)^2-u^2_E(r_*)}{1-u^2_E(r_*)}\ONE_{\chi(t)=r_*}dt,
\endaligned
$$
and
\begin{equation}
\label{velocity-of-characteristics-and-u}
 u(t)=\frac{\dot \chi(t)}{1-u_E^2(\chi(t))}.
\end{equation}
Observe in passing that $P^{t_0,t_1}(\chi)$ would vanish if the mass $M$ were zero, while the other contributions would reduce to the ones we already analyzed in the previous section. 
Of course, if $\chi$ does not visit $r_*$ on some time interval $(s_0,s_1)$, then, similarly to~\eqref{eq:integral-of-motion}, the constant   
\begin{equation}
\label{eq:conserved-quantity}
C=\frac{u^2(t)-u^2_E(\chi(t))}{1-u^2_E(\chi(t))}=\frac{u^2(t)- 2M/r}{1- 2M/r}\in 
\left[-\frac{2M/r}{1- 2M/r},1\right]= \left[-\frac{2M}{r-2M},1\right],
\end{equation}
arising  in the definition of $K^{t_0,t_1}(\chi)$, is a conserved quantity on $(s_0,s_1)$,
where we denoted $\chi(t)=r$ for brevity.

To any two points in the spacetime, we associate the action 
\[
\aligned
& A^{t_0,t_1}(r_0,r_1) :=\inf_{\chi:\chi(t_0)=r_0,\chi(t_1)=r_1}  \Big(
 K^{t_0,t_1}(\chi)+P^{t_0,t_1}(\chi)-B^{t_0,t_1}(\chi) \Big),
\quad  t_0<t_1,\quad r_0,r_1\ge r_*,
\endaligned
\]
and we denote by $\gamma^{s_0,s_1}_{x_0,x_1}$ the right-most path among those providing the minimum in this definition.
Finally, we introduce the action associated with any two points of the boundary
\[
 S^{t_0,t_1} :=A^{t_0,t_1}(r_*,r_*),\quad t_0<t_1.
\]
Using sub-additivity, we will prove the following result in Section~\ref{sec:stationary-Schw-proofs}.

\begin{proposition} 
\label{lem:rho}
There exists a scalar $\rho\le 0$ such that, with probability 1,
\begin{equation}
 \lim_{t_0\to-\infty}\frac{S^{t_0,t_1}}{t_1-t_0}=\rho.
\label{eq:linear-due-to-subadd}
\end{equation}
\end{proposition}

%-------------------------------------------------------------------------------------------------------------------------- 

We are now in a position to state our main results and we begin with an existence result.

\begin{theorem}[Existence of a global solution]
\label{thm:Schw-global-sol} 
Consider the Burgers--Schwarzschild equation with stationary random boundary condition $\phi$ given by~\eqref{eq:Eu-Schw-boundary-conds}. 
If $\rho \neq 0$ (equivalently, $\rho<0$), then there exists a stationary global solution.
\end{theorem}

\begin{corollary}
\label{Coro3} Let $q^2 :=\E \phi_+^2\in(0,1).$
If $q>u_E(r_*)$, then there exists a stationary global solution.
\end{corollary}

\begin{proof}[Proof of Corollary \ref{Coro3}]
If $q>u_E(r_*)$, then for the path $\chi$ that stays at $r_*$ at all times,
\begin{align*}
 \rho=&\lim_{t_0\to-\infty}\frac{S^{t_0,t_1}}{t_1-t_0}\le \lim_{t_0\to-\infty}\frac{-B^{t_0,t_1}(\chi)}{t_1-t_0}
\\=&-\lim_{t_0\to-\infty}\frac{1}{2(t_1-t_0)}\int_{t_0}^{t_1}\frac{\phi_+(t)^2-u^2_E(r_*)}{1-u^2_E(r_*)}dt=-\frac{1}{2}\frac{q^2-u^2_E(r_*)}{1-u^2_E(r_*)}<0,
\end{align*}
 so that Theorem~\ref{thm:Schw-global-sol} applies.
\end{proof}

Let us next state a result on the domain of attraction of the global solution constructed in Theorem~\ref{thm:Schw-global-sol}.
In analogy with the case of the (flat) Burgers equation, we can introduce a family of random evolution operators denoted by 
\[
 \Phi_{\omega}^{t_0,t_1}:\U\to\U\cap \D,\quad \omega\in\Omega,\ t_0<t_1,
\]
where we slightly abuse notation by using $\U$ and $\D$ for spaces of functions defined on
$[r_*,+\infty)$ instead of $\R_+$.
We also need to introduce
\[
 \V_{p} : =\bigcap_{p'<p} \U_{p'}=\left\{u\in L^1_\loc([r_*,+\infty)):\ \liminf_{x\to + \infty}\frac{\int_{r_*}^x u(x)dx}{x} \ge p\right\},\quad p\in(-1,1).
\]
In the theorem below, $d$ stands for the natural adaptation of the metric introduced in~\eqref{eq:metric-d} for functions on $\R_+$ to functions defined on $[r_*,+\infty)$

\begin{theorem}[Domain of attraction of the global solution]
\label{thm:attraction-Schw}
Suppose $\rho<0$, where $\rho$ was defined in \eqref{eq:linear-due-to-subadd}. Suppose $p\in\R$ satisfies one of the following two conditions:
 \begin{align}
 \label{eq:p_in_01-take-2}
  p&\in[0,1),\\
  \label{eq:p_in_-10-take-2}
  p\in(-1,0)\quad&\text{\rm and}\quad  \rho<-\frac{p^2}{2}.
 \end{align}
Then one has: 
\begin{enumerate}
 \item  For any $w\in \V_p$, with probability~$1$,
\begin{equation}
\label{eq:pull_attr_Schw-take-2}
 d(\Phi^{t,0}_{\omega} w, u_\omega)\to 0,\quad t\to-\infty.
\end{equation}
\item
If $\tilde u$ is a global stationary solution satisfying $\tilde u(0,\cdot)\in \V_p$ a.s., then $u(0,\cdot)=\tilde u(0,\cdot)$ a.s.   
\end{enumerate}
\end{theorem}

\medskip

We complete the program for the case $\rho<0$ with the following result on the behavior of the global 
solution at $+\infty$. For this we need to assume a condition on the averaging rate:
\begin{equation}
\label{eq:large-deviation-assumption-Schw}
\Pp(A(t,\eps)) < \alpha(t,\eps), 
\end{equation}
where
\[A(t,\eps)=\left\{\exists s>t:\quad  \left|\frac{1}{s}S^{0,s} - \rho \right|\ge \eps\quad \textrm{or}\quad  \left|\frac{1}{s} S^{-s,0} - \rho \right|\ge \eps \right\},\]
and, specifically, we assume that $\alpha$ is some function satisfying, 
for every $r>0$ and every $\eps>0$, 
\begin{equation}
  \label{eq:mixing-rate-Schw}
  \sum_{n=1}^{\infty} \alpha(nr,\eps) <+\infty.
 \end{equation}

\begin{theorem}\label{thm:u(plus-infty)-Schw} Suppose $\rho<0$ and assume that 
conditions~\eqref{eq:large-deviation-assumption-Schw}--\eqref{eq:mixing-rate-Schw} hold. Then, with probability~1, the global solution $u$ constructed in Theorem~\ref{thm:Schw-global-sol}  satisfies, with $\theta: = \sqrt{-2\rho}$, 
\begin{equation}
\label{eq:asymptotics-of-global-solution-take-2}
\lim_{r\to + \infty} u(0,r)=\theta. 
\end{equation} 
\end{theorem}

Similarly to the flat case, it is also possible to use the averaging rates for~$S^{t_0,t_1}$ to establish  estimates on the rates of convergence to the global solution.

\subsection{Proofs of the results in this section}\label{sec:stationary-Schw-proofs}
% of Theorems~\ref{thm:Schw-global-sol},~\ref{thm:attraction-Schw},~\ref{thm:u(plus-infty)-Schw}}
To prove Proposition~\ref{lem:rho}, we need some auxiliary estimates established in the next two lemmas.

\begin{lemma}\label{lem:P-subl} For every $t<0$, let $P^*(t)$ be the supremum of $P^{t,0}(\chi)$ over characteristics that do not touch the boundary $r_*$ between the times $t$ and $0$.
Then one has 
\[ 
\lim_{t\to -\infty}\frac{P^*(t)}{|t|}=0.
\]
\end{lemma}

\begin{proof}
For any $\chi$ and any $R>r_*$,
\begin{align*}
 P^{t,0}(\chi)&\le \int_t^0 \frac{2M}{\chi(s)-2M}\ONE_{\chi(s)\in(r_*,R]}ds + \int_t^0 \frac{2M}{\chi(s)-2M}\ONE_{\chi(s)\ge R}ds
\\& \le \Leb\{s:\chi(s)\in(r_*,R] \} \frac{2M}{r_*-2M} + |t|  \frac{2M}{R-2M},
\end{align*}
where $\Leb$ stands for the Lebesgue measure of a set. For any $R>0$ there is a constant $C(R)$ such that the time spent by any characteristic in $(r_*,R]$ before touching $r_*$
is bounded by $C(R)$. (This follows from the monotonicity of characteristics in the terminal or starting velocity and the fact that the characteristic started at $R$ with zero velocity spends a finite time
in $(r_*,R]$.) Therefore, to complete the proof of the lemma, it suffices to observe that $\frac{2M}{R-2M}$, the coefficient in front of $|t|$, on the r.h.s.\ can be made arbitrarily small by choosing $R$
sufficiently large.
\end{proof}

\begin{lemma}\label{lem:K-subl} For every $t<0$ and $r\ge r_*$, let $K_*(t,r)$ be the infimum of $K^{t,0}(\chi)$ over characteristics that do not touch $r_*$ between $t$ and $0$ and terminate at
$r$ at time $0$.
Then, one has 
\[ 
\lim_{t\to + \infty}\frac{K_*(t,r)}{|t|}=0.
\]
\end{lemma}
\begin{proof}
Let $\bar u(r,t)$ denote the value $u$ associated with the (unique) characteristic connecting $(t,r_*)$ to $(0,r)$. All characteristics arriving to $(0,r)$ with $u(0)>\bar u(r,t)$ reach $r_*$ between
$t$ and $0$, so they are irrelevant for the definition of~$K_*$.   
For the remaining characteristics, we get 
\[
 K^{t,0}(\chi)\ge \frac{1}{2}\frac{\bar u^2(r,t)-u^2_E(r)}{1-u^2_E(r)}|t|,
\]
and, since $\lim_{t\to-\infty} (\bar u^2(r,t)-u^2_E(r))=0$, the lemma follows.
\end{proof}

\begin{proof}[Proof of Proposition~\ref{lem:rho}]
The following obvious sub-additivity holds for all $\omega$:
\[
S^{t_0, t_2}\le S^{t_0,t_1} +S^{t_1,t_2}.
\]
Since all the terms in the definition of  $S^{t_0,t_1}$ allow for deterministic linear bounds,
we can apply the sub-additive ergodic theorem (a consequence of Birkhoff's ergodic theorem) 
to obtain that there is $\rho\in\R$ such that \eqref{eq:linear-due-to-subadd} holds.
To prove that $\rho\le 0$, it is sufficient, for each $t_0$, to introduce the characteristic $\chi_{t_0}$ that connects $(t_0,r_*)$ to $(t_1,r_*)$ and observe that the right-hand side of
\[
 S^{t_0,t_1}\le K^{t_0,t_1}(\chi_{t_0})+P^{t_0,t_1}(\chi_{t_0}),
\]
is sub-linear in $t_1-t_0$ due to Lemmas~\ref{lem:P-subl} and~\ref{lem:K-subl}.
\end{proof}

Now we turn to proving Theorem~\ref{thm:Schw-global-sol} and we begin with auxiliary statements.

\begin{lemma}\label{lem:minimizer-cannot-avoid-boundary} Suppose $\rho<0$ and let $r_1>r_*$ and $t_1\in\R$. Then with probability 1, there is $T_0<t_1$ such that for every $t_0<T_0$, $\gamma^{t_0t_1}_{r_*r_1}$ does not coincide with the characteristic connecting
$(t_0,r_*)$ to $(t_1,r_1)$ inside $(r_*,+\infty)$.
\end{lemma}
\begin{proof}
Since
\[
A^{t_0,t_1}(r_*,r_1)\le A^{t_0 t_1-1}(r_*,r_*)+A^{t_1-1t_1}(r_*,r_1), 
\]
we obtain that 
\[
 \limsup_{t_0\to-\infty}\frac{A^{t_0,t_1}(r_*,r_1)}{t_1-t_0}\le  \limsup_{t_0\to-\infty}\frac{A^{t_0, t_1-1}(r_*,r_*)}{t_1-t_0}+ \limsup_{t_0\to-\infty}\frac{A^{t_1-1,t_1}(r_*,r_1)}{t_1-t_0}=\rho<0.
\]
Our claim now follows from Lemma~\ref{lem:K-subl}.
\end{proof}

\begin{lemma}
\label{lem:constructing-global-solutions-take-2}
Suppose $\rho<0$. Then, for every $t_1$, with probability 1 and for every $R>r_*$, there is $T_0<t_1$ such that  all paths $\gamma^{t_0,t_1}_{r_*,r_1}$, $r_1\in [r_*, R]$, $t_0<T_0$, pass through the same point at some time between
$T_0$ and $t_1$, i.e. there exist $t'\in[T_0,t_1]$ and $r'>r_*$ such that for every $r_1\in [r_*, R]$ and every $t_0<T_0$,  we have $\gamma^{t_0,t_1}_{r_*,r_1}(t')=r'$.
\end{lemma}

\begin{proof} Due to the monotonicity of the paths  $\gamma^{t_0,t_1}_{r_*,r_1}$ w.r.t.\ the endpoint $r_1$, it is sufficient to prove that paths  $\gamma^{T_0,t_1}_{r_*,R}$ and $\gamma^{t_0,t_1}_{r_*,r_*}$ are guaranteed to intersect on $[T_0,t_1]$ for an appropriate choice of $T_0$.
The monotonicity property also implies that if these two paths do not intersect, then the last departure time~$t_*$ of $\gamma^{t_0,t_1}_{r_*,R}$ from~$r_*$ satisfies $t_*<T_0$. In particular, we obtain that $\gamma_{r_*,R}^{t_*,t_1}$ is the
characteristic connecting $(t_*,r_*)$ to $(t_1,R)$. However, this is impossible if $T_0$ is chosen according to Lemma~\ref{lem:minimizer-cannot-avoid-boundary}, which completes the proof.
\end{proof}

\begin{proof}[Proof of Theorem~\ref{thm:Schw-global-sol}]
Lemma~\ref{lem:constructing-global-solutions-take-2} implies that for any point $(t_1,r_1)$ the final segment of the minimizing path $\gamma_{r_*,r_1}^{t_0,t_1}$  is a characteristic curve starting at point $(t_*,r_*)$ for some $t_*$ and not visiting $r_*$
at any time between $t_*$ and $t$. This curve does not depend on $t_0$ if $t_0$ is  sufficiently close to $-\infty$. Therefore, using~\eqref{velocity-of-characteristics-and-u} and the velocity of this characteristic at space-time
location $(t_1,r_1)$ we can define $u(t_1,r_1)$ as 
\begin{equation}
\label{eq:global_u_Schw}
 u(t_1,r_1)=\frac{\dot \gamma_{r_*,r_1}^{t_0,t_1}}{1-u_E^2(\chi(t))}.
\end{equation}
thus obtaining a global solution. This solution is uniquely determined by the history of the boundary condition up to time $t$, and thus stationary.
\end{proof}

\begin{proof}[Proof of Theorem~\ref{thm:attraction-Schw}]
First of all, our assumptions \eqref{eq:p_in_01-take-2}--\eqref{eq:p_in_-10-take-2} can be rewritten as the inequality
\begin{equation}
\label{eq:requirement-in-1equation-take-2}
 H> \rho,
\end{equation}
where
\begin{equation}
\label{eq:def-of-minimum-take-2}
 H=H(p)=\min_{u\in[0,1]}\left\{  pu+\frac{1}{2}u^2\right\}= 
\begin{cases}
    0,& p\ge 0,\\
    -\frac{p^2}{2},& p<0.
   \end{cases}
\end{equation}

\smallskip

To prove the first part of the theorem, we need to show that for any $R>0$, there is $T_0<0$ such that for all $r_1\in[r_*,R]$ and all $t_0<T_0$, the optimal path $\chi$ arriving to $r_1$ at time $0$ and achieving the minimum in~\eqref{eq:rewriting-variational-principle-for-Schw} for $t_1=0$ has to visit~$r_*$ between $t$ and $0$. In other words, we must prove that characteristics that stay away from $r_*$ cannot be minimizers. 
Due to monotonicity of the minimizers with respect to the endpoint, it is sufficient to prove this for $r_1=R$. 
%Let us denote
%Then let us denote
%\[
% \Delta=\frac{1}{2}\frac{q^2-u_E^2}{1- u_E^2} - H.
%\]

Since the quantity $\frac{u^2-1}{1-2M/r}$ is preserved along the characteristics, we obtain that if $u_1:=u(0)<-u_E(r_1)$ for a characteristic $\chi$ and $u(\cdot)$ defined for $\chi$
via~\eqref{velocity-of-characteristics-and-u},
then
\begin{equation}
\label{eq:limiting-velocity-take-2}
 \frac{u_1^2-1}{1-2M/r_1}=u^2_\infty-1,
\end{equation}
where $u_\infty=u_\infty(u_1)=\lim_{t\to-\infty} u(t)<0$ is the asymptotic velocity along the characteristic. Therefore,
\begin{equation}
\label{eq:u_infty}
 u^2_\infty(u_1)=\frac{u_1^2-1}{1-2M/r_1}+1=\frac{u_1^2-1}{1-u_E^2(r_1)}+1=\frac{u_1^2-u_E^2(r_1)}{1-u_E^2(r_1)}=C,
\end{equation}
where $C$ is the conserved quantity defined in~\eqref{eq:conserved-quantity}, see also~\eqref{eq:limiting-speed}. 
Let us fix any $a\in(0,1)$ (a useful value for $a$ will be chosen later) 
and choose any $\hat u> u_E(r_1)$ such that $|u_\infty(-\hat u)|< a$.

If $|t_0|$ is sufficiently large, then any characteristic $\chi$ emitted in reverse time from $(t_1,r_1)=(0,R)$ with velocity $u_1\ge -\hat u$ either will reach $r_*$ between $t_0$ and $0$ or will satisfy
$\chi(t_0)< a|t_0|$. 
Since $w\in\V_p$, for every $\eps>0$ there is $c_\eps$ such that $W=\int_{r_*}^\cdot w$ satisfies
\begin{equation}
\label{eq:growth-of-initial-velocity-take-2}
 W(r)> c_\eps + (p-\eps)r,\quad r>r_*,
\end{equation}
and hence
\begin{equation}
\label{eq:growth-of-initial-velocity-take-2-inf-version}
 \inf_{r\in[r_*,r_0]}W(r)> c_\eps + ((p-\eps)\vee 0)r_0,\quad r_0>r_*.
\end{equation}
Therefore,
\begin{equation}
\label{eq:W-at-the-tip-of-chi}
W(\chi(t_0))>c_\eps+((p-\eps)\vee 0)a|t_0|. 
\end{equation}

Let us introduce the path $\gamma$ that coincides with $\gamma_{r_*,r_*}^{t_0,-1}$ between $t_0$ and $-1$ and goes straight from $(-1,r_*)$ to $(0,R)$.
For every $\eps>0$, with probability~$1$, for large enough $|t_0|$, the action of $\gamma$ satisfies
\begin{equation}
\label{eq:path-sticking-to-bdry-take-2}
A^{t_0,0}_W(\gamma)<  \left(\rho+\eps  \right)|t_0|,
\end{equation}
whereas the action of any charactersistic $\chi$ that does not touch the boundary between $t_0$ and $0$,
with $u_1\ge -\hat u$ (so $r_*<\chi(t_0)<a|t_0|$) satisfies
\[
A^{t_0,0}_W(\chi)> c_\eps+(p-\eps)a|t_0| + K_*(|t_0|,r_1).
\]
Lemma~\ref{lem:K-subl} implies that the last term is sub-linear in $|t_0|$, i.e. for any $\eps$, if $|t_0|$ is large enough,
its contribution is bounded by $\eps |t_0|$.

Therefore, by choosing $\eps>0$ and $a>0$ such that
\begin{equation*}
 \eps < -\frac{\rho}{4}
\end{equation*}
and
\[
 (p-\eps)a\vee 0>\frac{1}{4}\rho,
\]
we guarantee that
\[
 A^{t_0,0}_W(\gamma)<A^{t_0,0}_W(\chi),
\]
for large enough $|t_0|$ and $\chi$ satisfying $u_1\ge-\hat u$. Therefore, characteristics satisfying $u_1\ge -\hat u$ cannot be minimizers of $A_W^{t_0,0}$. Now we need to consider
fast trajectories emitted backwards with velocities satisfying $u_1<-\hat u$, and exclude this case.

We will need  uniform convergence to limiting velocities for unbounded characteristics. Due to~\eqref{eq:limiting-velocity-take-2}, along any orbit we have
\begin{equation}
\label{eq:rate-of-approach-to-u_infty}
 |u^2(t_0)-u_\infty^2(u_1)|=|u^2_\infty(u_1)-1| \, |u_E(\chi(t_0))|\le |u_E(\chi(t))|=\frac{2M}{\chi(t_0)}.
\end{equation}
Since the positions $\chi(t_0)$ are uniformly bounded below by the value of $\chi(t_0)$ corresponding
to $u_1=-\hat u$, we obtain, using~\eqref{eq:rate-of-approach-to-u_infty} and \eqref{velocity-of-characteristics-and-u} that
for every $\eps>0$ there is $T'_0<0$ such that if 
$u_1<-\hat u$ for a characteristic $\chi$ ending at~$r_1$ at time~$0$, then $|\dot \chi^2(t)-u^2_\infty|<\eps$
for all $t<T'_0$. Therefore, there is $T''_0<0$ such that
\begin{equation}
\label{eq:asymptotic-location-of-escape-chars-take-2}
 |\chi(t_0)-u_\infty(u_1)t_0|<\eps |t_0|,\quad t_0<T''_0,
\end{equation}
and the action $A_0^{t_0,0}$ (without the initial condition) of every path $\chi$ like that satisfies
\begin{equation}
\label{eq:asymptotic-action-of-escape-chars-take-2}
 \left|A_0^{t_0,0}(\chi)-\frac{u^2_\infty(u_1)}{2}|t_0|\right|\le \eps |t_0|, \quad t<T''_0.
\end{equation}
Due to~\eqref{eq:asymptotic-location-of-escape-chars-take-2} and \eqref{eq:growth-of-initial-velocity-take-2},
we obtain
\begin{equation}
\label{eq:initial-condition-at-char-landing-take-2} 
W(\chi(t_0))>c_\eps + (p-\eps)(|u_\infty(u_1)|-\eps\cdot \sign(p-\eps))|t_0|,\quad t_0<T''_0,
\end{equation}
so that the total action of any path like that satisfies
\[
A^{t_0,0}_W(\chi)> c_\eps + (p-\eps)(|u_\infty(u_1)|-\eps\cdot \sign(p-\eps))|t_0| + \frac{1}{2}(u^2_\infty(u_1)-\eps)|t_0|,\quad t_0<T''_0.
\]
Comparing this inequality with~\eqref{eq:path-sticking-to-bdry-take-2}, we see that it is now sufficient to check that 
\[
(p-\eps)(|u_\infty(u_1)|-\eps\cdot \sign(p-\eps)) + \frac{1}{2}(u^2_\infty(u_1)-\eps)>  \rho +\eps
\]
for sufficiently small $\eps$ and all $u_1<-\hat u$. Since $u_\infty$ is bounded, this follows by continuity from 
\[
 p |u_\infty(u_1)| + \frac{1}{2}u^2_\infty(u_1)>  \rho,
\]
which is a direct consequence of~\eqref{eq:requirement-in-1equation-take-2} and~\eqref{eq:def-of-minimum-take-2}. The proof of the first part is completed.

\medskip

The proof of the second part repeats the proof of the first one, except that when solving the initial-boundary value problem between $t_0$ and $0$, instead of the fixed initial condition $w$ (or its potential $W$) 
we use $\tilde u(t_0,\cdot)$ as the initial condition. This is possible since for any $\eps>0$ we can
define $\tilde U(t,r)=\int_{r_*}^r\tilde u(t,x)dx$ and
use the time stationarity of $\tilde u$ to replace~\eqref{eq:growth-of-initial-velocity-take-2} by
\[
\tilde U(t_0^{(k)},r)\ge c_\eps + (p-\eps)r,
\]
which holds true for a sequence of times $t_0^{(k)}\to-\infty$, and $r>r_*$. 
\end{proof}

\begin{proof}[Proof of Theorem~\ref{thm:u(plus-infty)-Schw}] 
%In this proof, for any unbounded forward orbit $\chi(t)$ we will denote $u_\infty=\lim_{t\to+\infty}u(t)>0$  the associated limiting velocity at $+\infty$ (unlike the proof of Theorem~\ref{thm:Schw-global-sol} where we worked with unbounded time-reversed orbits and used $u_\infty$ to denote the limiting value at $-\infty$).
For any $c\in(-2M/(r-2M),1]$, there are two characteristics with conserved quantity $C=c$ 
defined in \eqref{eq:conserved-quantity} visiting $r$ at time $0$, one with positive velocity
and one with negative velocity at~$r$.
The time of departure from $r_*$ of the former is denoted by $t_+(c,r)$ and of the latter by $t_-(c,r)$ (if exists).  There is only one characteristic passing through $r$ at time $0$ and corresponding to $c=-2M/(r-2M)$. Both times    $t_+(-2M/(r-2M),r)=t_-(-2M/(r-2M),r)$ 
are defined as the time of departure from $r_*$ of this characteristic.

We also recall that, according to \eqref{eq:u_infty}, if $c\ge 0$, then the associated characteristics $\chi_+(t)$ and $\chi_-(t)$ are unbounded and their velocities $u_+(t)$ and $u_-(t)$ satisfy
\[
\lim_{t\to\pm\infty}u_{\pm}(t) = \pm \sqrt{c}.
\]

It is sufficient to prove that, given any $\eps>0$, for sufficiently large $r\in\N$, the value 
$c$ associated with the optimal characteristic arriving to to $(0,r)$ is positive and  
$\sqrt{c}\in(\theta-\eps,\theta+\eps)$.

\begin{lemma}\label{lem:asymptotic-linear-time} For all $c_0\in(0,1)$, we have
\begin{equation*}
 \lim_{r\to + \infty}\sup_{c\in[c_0,1]}\left|\frac{t_+(c,r)}{r}+\frac{1}{\sqrt{c}}\right|=0.
\end{equation*}
\end{lemma}

\begin{proof}
For any $R>r_*$ the time spent by the characteristic $\chi_+$ inside $[r_*,R]$ is finite. 
The lemma follows since 
for any $\delta>0$, one can choose $R$ so that $|u_+(t)-\sqrt{c}|<\delta$
whenever $\chi(t)>R$.  
\end{proof}

Let us first exclude minimizers arriving to $r$ with the value of $c$ satisfying 
$\sqrt{c}>\theta+\eps$, i.e. $c>(\theta+\eps)^2$. For brevity, in this part of the proof, we write $t(c,r)$ for $t_+(c,r)$.

Due to Lemma~\ref{lem:asymptotic-linear-time},
%$w=u_\infty>\theta+\eps$. 
for sufficiently large $r$, 
\[
 t(\theta^2,r)<-\frac{r}{\theta+\eps/3},
\]
and, under the assumption we have made on $\sqrt{c}$,
%$w$,
\[
 t(c,r)>-\frac{r}{\theta+2\eps/3},
\]
so, 
if $\eps$ is sufficiently small, 
\begin{equation}
\label{eq:difference-in-departure-points-2}
 t(c,r)-t(\theta^2,r)\ge r \frac{\eps}{3(\theta+2\eps/3)(\theta+\eps/3)}\ge  r \frac{\eps}{4\theta^2}. 
\end{equation}

Denoting the characteristic associated with $t(c,r)$ by $\chi_{c}$, noticing that the value of $c$ obtained in~\eqref{eq:u_infty} is exactly the constant in front of the time factor in the expression for the flux in~\eqref{eq:variational-principle-for-Schw-def-of-F}, and using Lemma~\ref{lem:P-subl}, we can represent the actions of characteristics as
\begin{equation}
\label{eq:asymptotic-action-of-long-unbounded--theta}
A^{t(c,r),0}(\chi_{c})=\frac{c}{2}|t (c,r)|+\beta(c,r),
\end{equation}
where for any $c_0\in (0,1)$, 
\begin{equation}
 \lim_{r\to + \infty} \frac{\sup_{c\in[c_0,1]}|\beta(c,r)|}{r}=0
\label{eq:sub-linear-contrib-from-potential} 
\end{equation}
due to Lemma~\ref{lem:P-subl} and an obvious estimate $\limsup_{r\to + \infty} |t(c_0,r)|/r< + \infty$.
Either the total action of  
$\gamma_{r*,r*}^{t(\theta^2,r),t(c,r)}$ and $\chi_{c}$ is greater than the action of
 $\chi_{\theta^2}$, in which case $\chi_{c}$ cannot be a part of a minimizer, or 
\[
 A^{t(\theta^2,r),t(c,r)}(\gamma_{r*,r*}^{t(\theta^2,r),t(c,r)})
 +A^{t(w,r),0}(\chi_{c})\le A^{t(\theta^2,r),0}(\chi_{\theta^2}),
\]
which can be rewritten, using \eqref{eq:asymptotic-action-of-long-unbounded--theta}  as
\[
 S^{t(\theta^2,r),t(c,r)}+ \frac{c}{2}|t(c,r)|+\beta(c,r)\le
\frac{\theta^2}{2}|t (\theta^2,r)|+\beta(\theta^2,r),
\]
 i.e.
\begin{align}\notag
&
S^{t(\theta^2,r),t(c,r)}-\rho(t(c,r)-t(\theta^2,r))
\\ \notag
&\le 
-\rho(t(c,r)-t(\theta^2,r))-\frac{c}{2}|t(c,r)|-\beta(c,r) +\frac{\theta^2}{2}|t(\theta^2,r)|+\beta(\theta^2,r)
\\ \notag
&\le 
\frac{\theta^2}{2}(|t(\theta^2,r)|-|t(c,r)|)-\frac{c}{2}|t(c,r)|-\beta(c,r) +\frac{\theta^2}{2}|t(\theta^2,r)|+\beta(\theta^2,r)
\\ \notag
&\le  
\frac{\theta^2}{2}\left(\frac{r}{\theta}-\frac{r}{\sqrt{c}}\right)-\frac{c}{2}\frac{r}{\sqrt{c}} +\frac{\theta^2}{2}\frac{r}{\theta}+o(r)
\\ \notag
&\le \frac{r}{2}\left(2\theta-\frac{\theta^2}{\sqrt{c}}-\sqrt{c} +o(1)\right)
\\
&\le \frac{r}{2}\left(f(\sqrt{c}) +o(1)\right). \label{eq:dev-of-erg-ave-1}
\end{align}
Here, we used the definition of $\theta$ in the second inequality; $o(r)$ and $o(1)$, $r\to + \infty$, are uniform over  $c\in[c_0,1]$, due to Lemma~\ref{lem:asymptotic-linear-time} and \eqref{eq:sub-linear-contrib-from-potential}; 
and
\begin{equation}
\label{eq:auxiliary-f}
f(w)=2\theta-\frac{\theta^2}{w}-w.
\end{equation}
Since $f(w)$ decays for $w>\theta$, under our assumption $\sqrt{c}>\theta+\eps$, we obtain
\[
f(\sqrt{c})\le f(\theta+\eps)= 2\theta-\frac{\theta^2}{\theta+\eps}-(\theta+\eps)
=\frac{(\theta-\eps)(\theta+\eps)-\theta^2}{\theta+\eps}=-\frac{\eps^2}{\theta+\eps}\le - \frac{\eps^2}{2\theta}.
\]
Lemma~\ref{lem:asymptotic-linear-time} implies that for sufficiently large $r$,
\[
t(c,r)-t(\theta^2,r)\le 2 \left(\frac{1}{\theta} -\frac{1}{\sqrt{c}}\right)r \le \frac{4}{\theta}r,
\]
so using the last two displays and~\eqref{eq:dev-of-erg-ave-1}, we obtain for sufficiently large $r$,
\[
\frac{S^{t(\theta^2,r),t(c,r)}}{t(c,r)-t(\theta^2,r)}-\rho
\le \frac{-\frac{r}{2}\cdot\frac{\eps^2}{3\theta}}{\frac{4}{\theta}r}\le - \frac{\eps^2}{24}.
\]
Now, \eqref{eq:large-deviation-assumption-Schw} and \eqref{eq:difference-in-departure-points-2}
imply that the probability of this is bounded by 
$$
\alpha(r\eps/(4\theta^2), \eps^2/24).
$$
 Therefore,
the assumption \eqref{eq:mixing-rate-Schw} and
the Borel--Cantelli lemma imply that for sufficiently large $r\in\N$, the minimizer with endpoint $(0,r)$ cannot have $c>(\theta+\eps)^2$.

Let us fix $c_0\in(0,\theta^2/2)$ and exclude  minimizers with $c\in[c_0,(\theta-\eps)^2]$. Lemma~\ref{lem:asymptotic-linear-time} implies that for sufficiently large $r$, 
\[
 t(\theta^2,r)>-\frac{r}{\theta-\eps/3},
\]
and, under the assumption we have made on $c$, 
\[
 t(c,r)<-\frac{r}{\theta-2\eps/3}
\]
(due to monotonicity of characteristics with respect to $c$),
where $t(c,r)$ stands for either $t_+(c,r)$ or $t_-(c,r)$.
So 
\begin{equation}
\label{eq:difference-in-departure-points}
 t(\theta^2,r)-t(c,r)\ge r \frac{\eps}{3(\theta-2\eps/3)(\theta-\eps/3)}\ge  r \frac{\eps}{3\theta^2}. 
\end{equation}

Now, either the total action of  $\gamma_{r*,r*}^{t(w,r),t(\theta^2,r)}$ and $\chi_{\theta^2}$
is less than the action of  $\chi_c$, in which case~$\chi_c$ cannot be a minimizer, or 
\[
 A^{t(c,r),t(\theta^2,r)}(\gamma_{r*,r*}^{t(c,r),t(\theta^2,r)})+A^{t(\theta^2,r),0}(\chi_{\theta^2})\ge A^{t(c,r),0}(\chi_c),
\]
which can be rewritten as
\[
 S^{t(c,r),t(\theta^2,r)}+\frac{\theta^2}{2}|t(\theta^2,r)|+\beta(\theta^2,r)\ge \frac{c}{2}|t(c,r)|+\beta(c,r),
\]
 i.e.
\bel{eq:dev-of-erg-ave-2}
\aligned
&
S^{t(c,r),t(\theta^2,r)}-\rho(t(\theta^2,r)-t(c,r))
\\  
&\ge 
-\rho(t(\theta^2,r)-t(c,r))+\frac{c}{2}|t(c,r)|+\beta(c,r) -\frac{\theta^2}{2}|t(\theta^2,r)|-\beta(\theta^2,r)
\\  
&\ge 
\frac{\theta^2}{2}(|t(c,r)|-|t(\theta^2,r)|)+\frac{c}{2}|t(c,r)|+\beta(c,r) -\frac{\theta^2}{2}|t(\theta^2,r)|-\beta(\theta^2,r)
\\  
&\ge  
\frac{\theta^2}{2}\left(\frac{r}{\sqrt{c}}-\frac{r}{\theta}\right)+\frac{c}{2}\frac{r}{\sqrt{c}} -\frac{\theta^2}{2}\frac{r}{\theta}+o(r)
\\  
&\ge \frac{r}{2}\left(-2\theta+\frac{\theta^2}{\sqrt{c}}+\sqrt{c} +o(1)\right)
 \ge \frac{r}{2}\left(-f(\sqrt{c}) +o(1)\right), 
\endaligned
\ee
where $o(r)$ and $o(1)$, $r\to\infty$, are uniform over  $c\in[c_0,1]$, due to Lemma~\ref{lem:asymptotic-linear-time} and \eqref{eq:sub-linear-contrib-from-potential}, 
and $f$ is given in~\eqref{eq:auxiliary-f}. Since $f(w)$ grows for $w<\theta$, under our assumption $\sqrt{c}<\theta-\eps$, we obtain
\[
f(\sqrt{c})\le f(\theta-\eps)= 2\theta-\frac{\theta^2}{\theta-\eps}-(\theta-\eps)
=\frac{(\theta-\eps)(\theta+\eps)-\theta^2}{\theta+\eps}=-\frac{\eps^2}{\theta-\eps}\le - \frac{\eps^2}{\theta}. 
\]
Then, Lemma~\ref{lem:asymptotic-linear-time} implies that for sufficiently large $r$,
\[
t(\theta^2,r)-t(c,r)\le 2 \left(\frac{1}{\sqrt{c}}-\frac{1}{\theta}\right)r \le \frac{4}{c_0}r,
\]
so using the last two displays and~\eqref{eq:dev-of-erg-ave-2}, we obtain for sufficiently large $r$,
\[
 \frac{S^{t(w,r),t(\theta,r)}}{t(\theta,r)-t(w,r)}-\rho\ge \frac{\frac{r}{2}\frac{\eps^2}{\theta}}{\frac{4}{c_0}r}=\frac{c_0\eps^2}{8\theta}.
\]
Now, \eqref{eq:large-deviation-assumption-Schw} and \eqref{eq:difference-in-departure-points}
imply that the probability of this is bounded by 
$$
\alpha\left(\frac{r\eps}{3\theta^2}, \frac{c_0\eps^2}{8\theta}\right). 
$$
Therefore,
the assumption \eqref{eq:mixing-rate-Schw} and
the Borel--Cantelli lemma imply that for sufficiently large $r\in\N$, the forward unbounded minimizer with endpoint $(0,r)$ cannot have $c\in[c_0,\theta-\eps]$.

It remains to exclude values of $c< c_0.$ For this, we take an arbitrary large  $M>0$ and use 
Lemma~\ref{lem:asymptotic-linear-time} and monotone dependence of characteristics on the terminal velocity to 
choose $c_0$ and $R>0$
so that $|t(c,r)|>Mr$ for all $c<c_0$ and $r>R$.  Here $t(c,r)$ represents both $t_+(c,r)$ and  $t_-(c,r)$.
Denoting the corresponding path by $\chi_c$ we recall that $A^{t(c,r),0}(\chi_c)>0$. So, if $\chi_c$ is a minimizer, then 
\[
 S^{t(c,r),t(\theta^2,r)}+\frac{\theta^2}{2}|t(\theta^2,r)|+\beta(\theta^2,r)\ge A^{t(c,r),0}(\chi_c)>0,
\]
which implies
\begin{align*}
  S^{t(c,r),t(\theta^2,r)}-\rho(t(\theta^2,r)-t(c,r))&> -\rho(t(\theta^2,r)-t(c,r))-\frac{\theta^2}{2}|t(\theta^2,r)|-\beta(\theta^2,r)
  \\
  &=\frac{\theta^2}{2}|t(c,r)|-\beta(\theta^2,r),
\end{align*}
and
\[
\frac{S^{t(c,r),t(\theta^2,r)}}{t(\theta^2,r)-t(c,r)}-\rho>\frac{\frac{\theta^2}{2}|t(c,r)|-\beta(\theta^2,r)}{t(\theta^2,r)-t(c,r)}.
\]
Taking any $\delta>0$ and choosing $M$ and $c_0$  to ensure that 
$t(\theta^2,r)-t(c,r)>r$ and that
the right-hand side of the last display exceeds $\delta$, we obtain that, due to \eqref{eq:large-deviation-assumption-Schw}, existence of a minimizer associated with $t(c,r)$
has probability bounded by
$\alpha(r,\delta)$. Therefore,
the assumption \eqref{eq:mixing-rate-Schw} and
the Borel--Cantelli lemma imply that for sufficiently large $r\in\N$, the forward unbounded minimizer with endpoint $(0,r)$ cannot have $c<c_0$.
This completes the proof of Theorem~\ref{thm:u(plus-infty)-Schw}.
\end{proof}

%--------------------------------------------------------------------------------------------------------------------

\subsection{Ergodicity properties for the integrated fluid density}
\label{sec:Schw-stationary-transport}

All the above results in this section concern only the first component of the two-component 
system~\eqref{eq:ES}.  As for the second component, we know that $v$ is constant along characteristics, and that if $\rho<0$,
then there is a uniquely defined global foliation of space-time into characteristics that gives rise to the stationary
global solution $u$. Combining these facts we obtain that the only global solution $v$ compatible with
the global solution $u$ is constructed in the following way:

For every point $(t,r)$, there is a characteristic giving rise to the global solution $u(t,r)$. We denote by $t_*(t,r)$ the time when this characteristic leaves $r_*$ for the last time and set
\begin{equation}
\label{eq:global_v_Schw} 
v(t,x)=\psi(t_*(t,r)),\quad (t,x)\in\R\times (r_*,+\infty).
\end{equation}
This gives existence of the global solution  $(u_\omega, v_\omega)$ of the system~~\eqref{eq:ES}
and, under additional assumptions of Theorem~\ref{thm:attraction-Schw} 
we also obtain 1F1S (uniqueness and attraction) for these global solutions,
analogously to Theorem~\ref{thm:v}:
\begin{theorem}\label{thm:v-Schw} 
Consider the pressureless Euler--Schwarzschild system~\eqref{eq:Eu-Schw}
in $[r_*, +\infty)$ with stationary random boundary conditions given by ~\eqref{eq:Eu-Schw-boundary-conds}.
\begin{enumerate}  
\item (Existence.) If $\rho<0$, where $\rho$ has been defined in~\eqref{eq:linear-due-to-subadd}, then
equations~\eqref{eq:global_u_Schw} and \eqref{eq:global_v_Schw} define a global stationary solution $(u_\omega, v_\omega)$.
\item (Uniqueness.) If, additionally, a number $p\in\R$ satisfies one of the 
conditions~\eqref{eq:p_in_01-take-2} or~\eqref{eq:p_in_-10-take-2}  holds, then
every stationary global solution $(\tilde u_\omega, \tilde v_\omega)$ of the system~\eqref{eq:Eu-Schw} such that $u_\omega(t,\cdot)\in \V_p$ coincides with $(u_\omega, v_\omega)$ with probability~1.
  \item (Pullback attraction.) Under the conditions of part 2 of this theorem, the solution $(u_\omega, v_\omega)$ is a random one-point pullback attractor, i.e., with probability 1, for every pair $(w,w')$  with $w\in\V_{p}$ and every $r>r_*$, there is $t_0<0$
  such that for every $t<t_0$, the solution of the pressureless Euler system with initial condition $(w,w')$ assigned at $t$, coincides with $(u_\omega,v_\omega)$ on $\{0\}\times(r_*,r]$.
\end{enumerate}
\end{theorem}

%============================================================================

\section{Concluding remarks}
\label{sec7}

\subsection{Controlling the time of return to a given location}

We conclude this paper with several additional observations. First, we observe that the lower bound in \eqref{eq:timeofreturn} below vanishes when the path begin with a vanishing velocity and tends to infinity as the initial velocity approaches the escape velocity. 

\begin{lemma} Let $\chi$ be a solution to the characteristic equation defined in some interval $[t_0, t_1]$ and initiating at some point $(t_0, r_0)$ with $r_0 \geq r_* > 2M$ at a positive and sub-escape velocity $u_0 < u_0^E$. Suppose that this path return back to the same location at time $t_1$, that is suppose that $\chi(t_1)=\chi(t_0)=r_0$. Then one has the inequality
\bel{eq:timeofreturn}
t_1 - t_0 \geq 4M u_0 {1 - u_0^2 \over (u_0^E)^2 - u_0^2}.
\ee
\end{lemma}

\begin{proof} We consider the point $(t_2, r_2)$  at which the path changes its direction, that is, when $\chi'(t_2) = 0$ with $t_2 \in (t_0, t_1)$ and $r_2 > r_0$. From the ODE satisfied by $\chi(t)$ we find that $\chi'(t)$ vanishes precisely when $1 - (1 - 2M/r_2) {1 - u_0^2 \over 1 - (u_0^E)^2} = 0$, which yields us 
$$
r_2 = 2M {1 - u_0^2 \over (u_0^E)^2 - u_0^2}, 
\quad \text{ therefore } \quad 
r_2 - r_0 = 2M u_0^2 {1 - (u_0^E)^2 \over (u_0^E)^2 - u_0^2}.
$$
Moreover, from the same ODE we also find
$$
r_2 - r_0 \leq (t_2 - t_0) {1 - (u_0^E)^2 \over 1- u_0^2} u_0, 
$$
hence the time for the path to change its direction is estimated as
$$
t_1 - t_0 \geq 2M u_0 {1 - u_0^2 \over (u_0^E)^2 - u_0^2}.
$$
The same estimate also holds for the time needed to return to its initial location. 
\end{proof}

%----------------------------------------------------------------------------------------------------------------------------------- 

\subsection{The Burgers equation with moving boundaries}

In the case that one cannot control the whole of the spacetime, it could useful to have a notion of solution defined in a moving domain, and we present such a concept for the standard Burgers equation. These {\sl solutions with a boundary} are defined in a compact interval only and the boundary of the interval ``moves'' at the characteristic speed. 
Observe that the solution is {\sl not defined outside} the domain of interest. 
By including both the effects of the curved geometry and a boundary condition, a similar definition and theorem can be stated for the Euler--Schwarzschild equation with moving boundary. 

\begin{definition} Consider the classical  Burgers equation. A {\em weak solution with moving boundaries} 
 is a function $u=u(t,x) \in \RR$ defined for all $t \geq 0$ and for $x \in (\phi_0(t), \phi_1(t)$ such that 
\bei

\item $u=u(t,x)$ has bounded variation in $x \in [0,1]$. 

\item $u_t + (u^2/2)_x=0$ holds in the weak sense in the domain of definition. 

\item $\phi_0(0) = 0$ and $\phi_1(0) = 1$. 
 
\item $\varphi_0'(t)= u(t, \phi_0(t))$ and $\varphi_1'(t)= u(t, \phi_1(t))$ for all but countably many times $t>0$. 

\eei 
\end{definition}

\begin{theorem}[Existence of weak solutions with moving boundaries]
 The Cauchy problem for the Burgers equation with prescribed initial data $u_0 \in L^1(0,1)$, that is 
$$
u(0,x)= u_0(x), \quad x \in [0,1]
$$
admits a `weak solution with moving boundaries', which is determined by the explicit formula 
\be
u(t,x)= {x-y(t,x) \over t}, \qquad \phi_0(t) \leq x \leq \phi_1(t), 
\ee 
where for each time $t>0$ and each $x \in \RR$, one denotes by $y(t,x)$ any point achieving the minimum of
\be
y \in [0,1] \mapsto \int_0^y u_0 \, dz + {(x-y)^2 \over 2t}
\ee
and the solution boundaries are defined by 
\be
\aligned
\phi_0(t) & = \max \big\{  x \, / \,     y(t, x') =0  \text{ for all } x'<x             \big\}, 
%y(t, \phi_0(t)-) = 0, 
%\qquad y(t, \phi_1(t)+) = 1.  
\\
\phi_1(t) &= \min \big\{  x \, / \,     y(t, x') =1  \text{ for all } x'>x             \big\}.
\endaligned
\ee
\end{theorem} 

\begin{proof} 
Observe that $y$ varies in the interval $[0,1]$ only, but that $x$ is allowed to describe the real line.  
It is clear a minimizer $y(t,x) \in [0,1]$ always exists and the formula yields a weak solution to the Burgers equation. Also, for each fixed time $t$, we have $y(t,x) =0$ if $x$ is sufficiently negative, while 
 $y(t,x) =1$ if $x$ is sufficiently positive.  We then define the support of the solution as stated in the theorem. 
\end{proof}

%----------------------------------------------------------------------------------------------------------------------------------- 

\subsection{The Schwarzschild--Burgers equation in the outer domain of communication} 

Rather than solving in the ``half-space'' $r >r_*$ for some $r_* >2M$ as we did in Section~\ref{sec35}, one could solve the initial value problem for the Schwarzschild-Burgers equation in the outer domain of communication, that is, in the entire region $r>2M$. For this problem, {\sl  no boundary condition} is required on the boundary $r=2M$, which is to be expected ---since this boundary is the ``horizon'' of the Schwarzschild black hole from which no ``information'' can escape.  We thus restate Theorem~\ref{theo31} in the following form.  
The explicit formula therein is virtually the same, but with all of the boundary terms suppressed, and we do not repeat it here.

\begin{theorem}[Existence theory in the outer domain of communication]
\label{theo32}
Given any initial data $u_0: (2M, +\infty) \mapsto (-1,1)$ prescribed at some time $t_0 \in \RR$ and satisfying the integrability condition $(1-2M/r) \in L_\loc^1([2M, +\infty))$, the associated initial value problem for  the Schwarzschild-Burgers equation in the outer domain of communication admits a solution with locally bounded variation which is defined globally for all $t  \geq t_0$ and $r>2M$.
\end{theorem}

%=========================================================================== 

\bibliographystyle{plain-initials}
\bibliography{fluid-with-b}

\end{document}